\documentclass{amsart}
\usepackage[all]{xy}
\usepackage{graphicx}
\usepackage{enumitem}
\usepackage{amscd}
\usepackage{mathrsfs}
\usepackage{amsmath}
\usepackage{amssymb}
\usepackage[OT2,T1]{fontenc}
\usepackage{commath}
\usepackage{tikz-cd}
\tikzset{
    labl/.style={anchor=north, rotate=90}
}

\usepackage{adjustbox}

\usepackage{xr}

\usepackage{thmtools}
\usepackage[colorlinks]{hyperref}
\usepackage{cleveref}
\crefformat{equation}{(#2#1#3)}
\crefname{section}{\S}{\S\S}
\crefname{subsection}{\S}{\S\S}

\numberwithin{equation}{subsection}

\newtheorem{theorem}[equation]{Theorem}
\newtheorem{lemma}[equation]{Lemma}
\newtheorem{proposition}[equation]{Proposition}
\newtheorem*{theorem*}{Theorem}
\newtheorem{corollary}[equation]{Corollary}
\newtheorem*{corollary*}{Corollary}
\newtheorem*{proposition*}{Proposition}

\theoremstyle{remark}
\newtheorem{remark}[equation]{Remark}
\newtheorem{example}[equation]{Example}
\newtheorem{claim}[equation]{Claim}
\newtheorem*{claim*}{Claim}

\theoremstyle{definition}
\newtheorem{definition}[equation]{Definition}

\DeclareMathOperator{\Spec}{Spec}
\DeclareMathOperator{\gp}{gp}
\DeclareMathOperator{\rk}{rk}
\DeclareMathOperator{\triv}{triv}
\DeclareMathOperator{\sat}{sat}

\DeclareMathOperator{\td}{tr.deg}
\DeclareMathOperator{\cok}{cok}
\DeclareMathOperator{\Hom}{Hom}

\DeclareMathOperator{\Tor}{Tor}

\DeclareMathOperator{\pr}{pr}
\DeclareMathOperator{\depth}{depth}
\DeclareMathOperator{\D}{D}
\DeclareMathOperator{\PD}{pd}

\title{Toric flatness {I}}
\author[R. Lodh]{Rémi Lodh}
\email{remi.shankar@gmail.com}
\hypersetup{hypertexnames=false}

\begin{document}
\begin{abstract}
We establish some basic properties of \emph{t-flat} fine log schemes, i.e., fine log schemes whose underlying schemes are flat over their log structure (in a suitable sense). Notably, we show that the depth of a t-flat fs log scheme is simply determined by its logarithmic stratification. We also show that there is a good theory of log cotangent complexes for t-flat fine log schemes, and use the log cotangent complex to give a homological criterion for t-flatness.
\end{abstract}
\maketitle

\section{Introduction}
In \cite{logreg}, Kazuya Kato extended regularity to log schemes. Known as log regularity, the notion he introduced proved a crucial tool for working with singular schemes, notably being used by Gabber \cite{travauxGabber} to obtain fundamental results in \'etale cohomology. Log regularity is a strong condition: it implies very good properties of the underlying scheme, even desingularisation. However, like regularity, it is rather rigid; for instance, it does not localise in the flat topology.

In his paper, Kato showed that log regular singularities are particularly well behaved as modules over their log structure---we would say that they are flat over their log structure---and used this fact to establish basic properties. We propose here to isolate this feature of log regularity and study it for its own sake. In other words, to study those log schemes whose underlying schemes are flat, in the monoidal sense, over the log structure. Following Gabber, one calls such log schemes \emph{t-flat}, where the \emph{t} stands for \emph{toric}. The notion is not new: it was first defined and studied by Gabber over a quarter of a century ago (unpublished, cf. \cite{toricflattening}) and then used by Thompson \cite{thompson} to extend log regularity to non-saturated log schemes; later, foundations in the monoidal setting were laid in the book of Gabber and Ramero \cite{gabberramero}.

The theme of this paper is that \emph{t-flatness} in logarithmic geometry is a natural, well-behaved condition that strengthens analogies between schemes and log schemes. It forms a bond between a log scheme and its logarithmic stratification, reducing the study of certain local properties of the former to the latter. Here are some of its properties.

\begin{theorem*}
For fine log schemes with locally noetherian underlying schemes, t-flatness has the following properties.
\begin{enumerate}[label=\normalfont(\roman*)]
\item It is open in the \'etale topology.
\item It is stable under log flat base change.
\item For fs log schemes, it is local for the (Kummer) log flat topology.
\item The log cotangent complex of t-flat fine log schemes satisfies all of the desired properties.
\end{enumerate}
Let $(X,M_X)$ be a t-flat fs log scheme with $X$ locally noetherian.
\begin{enumerate}[label=\normalfont(\roman*),resume]
\item If $\mathcal{O}(X)$ contains a field, then $X$ is \'etale-locally flat over a toric variety.
\item The depth of $X$ is simply determined by the depth of its logarithmic stratification. For instance, $X$ is Cohen--Macaulay if and only if its logarithmic strata are so.
\end{enumerate}
\end{theorem*}

Notably, it follows from (iii) and (v) that the depth of t-flat fs log schemes behaves similarly in the log flat topology as it does in the classical flat topology: the depth of a covering is at least that of the target. In particular, under t-flatness, a quasi-finite Kummer log flat covering of a Cohen--Macaulay scheme must itself be Cohen--Macaulay.

More generally, we define notions of logarithmic dimension and logarithmic depth and show they behave in the log flat topology like the classical notions in the flat topology. Crucially, for t-flat fs log schemes these notions coincide with the usual notions of dimension and depth, thereby providing a means of studying the latter via logarithmic geometry. For example, we show that the depth of a t-flat fs log scheme is stable under log blowups.

Concerning (iv), we show that the two log cotangent complexes studied in \cite{logcotangent} coincide for morphisms of fine log schemes with t-flat target. Thus, log cotangent complexes of morphisms with t-flat targets have properties analogous to the classical cotangent complex. Furthermore, we characterise t-flatness by the vanishing of the homology of a log cotangent complex---in perfect analogy with the classical cotangent complex characterisation of a regular closed immersion---and use this to recover a recent homological criterion for log regularity due to Conde-Lago and Majadas \cite{cotangentlogreg}.

%\tableofcontents

\subsection*{Overview}
In \cref{monoidalflatnesssection}, we first state the results on flatness over monoids from \cite{gabberramero} that we need in the paper, suitably formulated so as to be later applicable to fine log schemes. Then, we prove inequalities for the dimension and depth of noetherian local rings and give conditions for equality. The depth (in)equality (\Cref{depthequality}) appears to be new. Finally, we show that, under certain conditions, a noetherian local ring is flat over a sharp fine monoid if and only if it is flat over a ring that is close to being log regular (\Cref{localpresentation}).

In \cref{tflatsection}, we define t-flatness and prove some basic properties. In particular, we study the relationship between t-flatness and log flatness and show that, over a field with trivial log structure, a log scheme is t-flat if and only if it is log flat (\Cref{tlflogflat}). Then, we define \emph{logarithmic dimension} and \emph{logarithmic depth} and generalise the inequalities obtained in \cref{monoidalflatnesssection} to t-flat fine log schemes. Notably, in \Cref{ldimineq} (resp. \Cref{depthequality2}) we show that the logarithmic dimension (resp. logarithmic depth) of a t-flat fine (resp. fs) log scheme coincides with the usual dimension (resp. depth). From this we deduce a close relationship between the depth of a t-flat fs log scheme and its logarithmic stratification (cf. \Cref{CM}).

In \cref{logcotangentsection}, we first show that the Gabber and Olsson log cotangent complexes coincide for a morphism of fine log schemes with t-flat target (\Cref{GabberOlsson}). Then, in \Cref{distinguished}, we use a log cotangent complex to characterise t-flatness. The statement can be interpreted as saying that an fs log scheme is t-flat if and only if the inclusion of its logarithmic stratification is a `log regular immersion' (however, we do not define such a notion). From this result, we recover the characterisation from \cite{cotangentlogreg} of log regularity in terms of the vanishing of the homology of a log cotangent complex (\Cref{logregvanishing}), and also a version for \emph{log complete intersections} (\Cref{lcivanishing}). Finally, we use t-flatness to define the notion of \emph{t-lci} log scheme, which for fs log schemes of finite type over a field is equivalent to being log syntomic in the sense of \cite{niziolKlog}.

\subsection*{Acknowledgements} As I hope to have made clear above, the central concept of this paper is due to Ofer Gabber. I merely develop it here, relying crucially on some results from his book with Lorenzo Ramero \cite{gabberramero}. I thank Ramero for invaluable comments on \cref{monoidalflatnesssection}, Gabber for sharing his notes \cite{toricflattening}, and Luc Illusie for encouraging correspondence.

\subsection{Conventions and notation}\label{conventions}
All rings and monoids are commutative with unity. We usually write the monoid operation multiplicatively with identity element $1$, except for the standard monoid $\mathbb{N}$ and its derivatives. For a monoid (resp. sheaf of monoids) $P$, we denote by $P^*$ the subgroup (resp. subsheaf) of invertible elements and write $\overline{P}:=P/P^*$. If $A$ is a finitely generated abelian group, its rank is denoted $\rk A$.

For a ring $R$ and set $S$ we write $R^{(S)}$ for the free $R$ module on basis $S$ and denote $[s]\in R^{(S)}$ the basis element defined by $s\in S$. If $S$ is a monoid, then $R^{(S)}=R[S]$ is an $R$-algebra. If $S$ is a sharp monoid, then we denote by $R[[S]]$ the completion of $R[S]$ by the ideal generated by the elements $[s]$ for $s\in S\setminus\{1\}$.

Given a monoid $P$, we write $\D(P)$ for the log scheme consisting of the scheme $\Spec(\mathbb{Z}[P])$ and the natural log structure induced by $P$.\footnote{this notation, from \cite{illusieoverview}, differs from that of \cite{logbook}} If $P$ is a group, then $\D(P)$ is the usual diagonalisable group scheme with trivial log structure. Note that $\D(\mathbb{N})$ represents the contravariant functor on the category of log schemes defined by the log structure, i.e., $\Hom((X,M),\D(\mathbb{N}))=M(X)$ for a log scheme $(X,M)$ (cf. \cite[III, 1.2.4]{logbook}).

Schemes can be viewed as log schemes with the trivial log structure; in that case we will often simply omit the log structure from notation. Unless explicitly stated to the contrary, when speaking of fine (resp. fs) log schemes, categorical constructions such as fibre products are made in the category of fine (resp. fs) log schemes.

For a log scheme $(X,M_X)$, we denote by $(X,M_X)_{\triv}\subset X$ the subset of points $x\in X$ such that $M_{X,\bar{x}}\cong\mathcal{O}_{X,\bar{x}}^*$ for any geometric point $\bar{x}\to X$ lying above $x$. If $M_X$ is coherent, then $(X,M_X)_{\triv}\subset X$ is open (\cite[III, 1.2.8]{logbook}). We will usually drop $M_X$ from the notation and write $X_{\triv}=(X,M_X)_{\triv}$ when there is no ambiguity about $M_X$.

We define Kummer homomorphisms of integral monoids (resp. fine log schemes) as in \cite[I, 4.3.1]{logbook} (resp. \cite[III, 2.3.2]{logbook}). In particular, we do not assume the monoids (resp. log schemes) to be necessarily saturated.

If $X$ is a scheme and $S$ the spectrum of a separably closed field, we denote a geometric point $x:S\to X$ simply by $x\to X$. We write $X_{(x)}=\Spec(\mathcal{O}_{X,x})$ for the strict localisation of $X$ at $x$, and denote by $k(x)$ the residue field of $\mathcal{O}_{X,x}$.\footnote{the notation $k(x)$ is nonstandard but should otherwise not cause any issue} By \emph{neighbourhood of $x$}, we mean an \'etale neighbourhood. If $f:X\to Y$ (resp. $f:(X,M_X)\to (Y,M_Y)$) is a morphism of schemes (resp. log schemes), then we denote by $f(x)$ the geometric point $x\to X\overset{f}{\to}Y$ and call it the \emph{image} of $x$ (under $f$). The transcendence degree of the extension $k(f(x))\subset k(x)$ is denoted $\td (k(x)/k(f(x)))$.

\section{Monoidal flatness}\label{monoidalflatnesssection}

\subsection{Pointed monoids}
Let $P$ be a monoid. A \emph{pointed $P$-module} is a $P$-module $M$ (i.e., a $P$-set in the terminology of \cite{logbook}) with a distinguished \emph{zero element} $0_M\in M$ satisfying $p0_M=0_M$ for all $p\in P$. Pointed $P$-modules form a category, the morphisms being $P$-module morphisms that preserve the zero elements. Given a pointed sub-$P$-module $N\subset M$, one defines the quotient pointed $P$-module $M/N$ as the colimit, in the category of pointed $P$-modules, of the diagram $0\leftarrow N\rightarrow M$. Its underlying set is just the colimit of the diagram in the category of sets, which can be identified with $M\setminus (N\setminus\{0_N\})$.

$P$ is a \emph{pointed monoid} if it is pointed as a module over itself. An \emph{ideal} of a pointed monoid $P$ is a pointed $P$-submodule $J\subset P$; $J$ is \emph{prime} if $P\setminus J$ is a monoid. Given an ideal $J\subset P$, the quotient pointed $P$-module $P/J$ is itself a pointed monoid.

Given a pointed monoid $P$ and ring $R$, one defines
\begin{equation}\label{pointedalgebra1}
R\langle P\rangle :=R[P]/([0_P])
\end{equation}
If $J\subset P$ is an ideal, then (cf. \cite[6.1.31]{gabberramero})
\begin{equation}\label{pointedalgebra2}
R\langle P/J\rangle\cong R[P]/JR[P]
\end{equation}

Given a non-pointed monoid $P$, one can formally adjoin a zero element $0_P$ to obtain a pointed monoid $P_{\circ}$ (\cite[4.8.14]{gabberramero}), and similarly for ideals $J\subset P$.

\begin{lemma}\label{idealcorrespondence}
The map $J\mapsto J_{\circ}$ induces a bijection on the sets of ideals of $P$ and ideals of $P_{\circ}$, preserving prime ideals.
\end{lemma}
\begin{proof}
Given an ideal $K\subset P_{\circ}$, $J:=K\cap P$ is an ideal of $P$ (possibly empty) and $K=J\cup\{0_P\}$, so the map $J\mapsto J_{\circ}$ is bijective. Moreover, $K$ is prime if and only if $P_{\circ}\setminus K=P\setminus J$ is a monoid, i.e., if and only if $J$ is prime.
\end{proof}

We prove a lemma for later use.

\begin{lemma}\label{torindlemma}
Let $P$ be a pointed monoid and $R$ be a ring.
\begin{enumerate}[label=\normalfont(\roman*)]
\item The canonical morphism $R\otimes_{\mathbb{Z}}^L\mathbb{Z}\langle P\rangle\to R\langle P\rangle$ is a quasi-isomorphism.
\item For any homomorphism of pointed monoids $P\to Q$, the canonical morphism $R\langle P\rangle\otimes^L_{\mathbb{Z}\langle P\rangle}\mathbb{Z}\langle Q\rangle\to R\langle Q\rangle$ is a quasi-isomorphism.
\end{enumerate}
\end{lemma}
\begin{proof}
(i) follows from the fact that $\mathbb{Z}\langle P\rangle$ is a free $\mathbb{Z}$-module and $R\otimes_{\mathbb{Z}}\mathbb{Z}\langle P\rangle\cong R\langle P\rangle$. This implies
\begin{align*}
R\langle P\rangle\otimes^L_{\mathbb{Z}\langle P\rangle}\mathbb{Z}\langle Q\rangle &\cong R\otimes_{\mathbb{Z}}^L\mathbb{Z}\langle P\rangle\otimes^L_{\mathbb{Z}\langle P\rangle}\mathbb{Z}\langle Q\rangle \\
& \cong R\otimes_{\mathbb{Z}}^L\mathbb{Z}\langle Q\rangle \\
&\cong R\langle Q\rangle
\end{align*}
proving (ii).
\end{proof}

See \cite[\S 4.8 and \S 6.1]{gabberramero} for details about pointed monoids and modules.

\subsection{Flatness over a monoid}
Our discussion of monoidal flatness will be limited to integral (non-pointed) monoids $P$ and $\mathbb{Z}[P]$-modules.

\begin{definition}\label{Pflatdef}
Let $P$ be an integral monoid and $M$ a $\mathbb{Z}[P]$-module. We say $M$ is \emph{$P$-flat} if, for any ideal $J\subset P$, the canonical map
\begin{equation}\label{idealtensormap} (J)\otimes_{\mathbb{Z}[P]}M\to M
\end{equation}
is injective, where $(J)\subset\mathbb{Z}[P]$ is the ideal generated by $\{[j]:j\in J\}$.
\end{definition}

\begin{lemma}\label{nonpointedPflat}
Let $P$ be an integral monoid and $M$ a $\mathbb{Z}[P]$-module. The following are equivalent.
\begin{enumerate}[label=\normalfont(\roman*)]
\item $M$ is $P$-flat.
\item $M$ is $\varphi$-flat in the sense of \cite[6.1.34]{gabberramero}, where $\varphi:P_{\circ}\to\mathbb{Z}\langle P_{\circ}\rangle\cong\mathbb{Z}[P]$ is the canonical homomorphism.
\end{enumerate}
\end{lemma}
\begin{proof}
Indeed, by \cref{pointedalgebra2}, for any ideal $J\subset P$ we have $\mathbb{Z}[P]/(J)\cong\mathbb{Z}\langle P_{\circ}/J_{\circ}\rangle$, so the result follows from \cite[6.1.40]{gabberramero} via \Cref{idealcorrespondence}.
\end{proof}

\begin{remark}
In \cite[\S6]{gabberramero}, monoidal flatness is developed in a more general context. For our purposes, the situation of \Cref{Pflatdef} will suffice, with \Cref{nonpointedPflat} allowing us to make use of results of \cite{gabberramero}.
\end{remark}

\begin{proposition}\label{primeideals}
Let $P$ be an integral monoid, $R$ a ring and $M$ an $R[P]$-module. The following are equivalent.
\begin{enumerate}[label=\normalfont(\roman*)]
\item $M$ is $P$-flat.
\item For every injective homomorphism of integral monoids $P\to Q$ and ideal $K\subset Q$, $\Tor_i^{R[P]}(M,R[Q]/KR[Q])=0$ for $i>0$.
\item For every ideal $J\subset P$, $\Tor_i^{R[P]}(M,R[P]/JR[P])=0$ for $i>0$.
\end{enumerate}
If $P$ is fine, then this is also equivalent to
\begin{enumerate}[label=\normalfont(\roman*),resume]
\item the map \textup{\cref{idealtensormap}} is injective for all prime ideals $J\subset P$.
\end{enumerate}
\end{proposition}

\begin{proof}
In (ii), consider the quotient pointed monoid $Q_{\circ}/K_{\circ}$. By \Cref{torindlemma}, we have $R\langle Q_{\circ}/K_{\circ}\rangle\cong R\langle P_{\circ}\rangle\otimes_{\mathbb{Z}\langle P_{\circ}\rangle}^L\mathbb{Z}\langle Q_{\circ}/K_{\circ}\rangle$, hence
\[ M\otimes_{R\langle P_{\circ}\rangle}^LR\langle Q_{\circ}/K_{\circ}\rangle\cong M\otimes_{\mathbb{Z}\langle P_{\circ}\rangle}^L\mathbb{Z}\langle Q_{\circ}/K_{\circ}\rangle \]
Thus, by \cref{pointedalgebra2}, in (ii) we may assume $R=\mathbb{Z}$. Similarly for (iii).

Now note that, since $P\to Q$ is injective and $Q$ is integral, $Q_{\circ}$ is an integral pointed $P_{\circ}$-module in the sense of \cite[6.1.21]{gabberramero}, hence so is $Q_{\circ}/K_{\circ}$ (cf. \cite[6.1.27 (iii)]{gabberramero}). If (i) holds, then, by \Cref{nonpointedPflat}, $M$ is $\varphi$-flat so, by \cite[6.1.40]{gabberramero}, we have $\Tor_i^{\mathbb{Z}\langle P_{\circ}\rangle}(M,\mathbb{Z}\langle Q_{\circ}/K_{\circ}\rangle )=0$ for $i>0$, hence (ii) (cf. \cite[6.1]{logreg}). Clearly, (ii) $\Rightarrow$ (iii) and, since (iii) $\Rightarrow$ (i) by definition, this proves the equivalence of (i)--(iii). The equivalence with (iv) in the fine case follows again from \cite[6.1.40]{gabberramero}.
\end{proof}

\begin{proposition}\label{monoidflatness}
Let $P\to Q$ be an injective homomorphism of integral monoids and $M$ a $\mathbb{Z}[P]$-module. Consider the following assertions.
\begin{enumerate}[label=\normalfont(\roman*)]
\item $M$ is $P$-flat.
\item $M\otimes_{\mathbb{Z}[P]}\mathbb{Z}[Q]$ is $Q$-flat.
\end{enumerate}
We have \textup{(i)} $\Rightarrow$ \textup{(ii)}. If $P$ and $Q$ are fine and $P\to Q$ is Kummer,\footnote{recall that we do not necessarily assume the monoids in Kummer homomorphisms to be saturated} then \textup{(ii)} $\Rightarrow$ \textup{(i)}.
\end{proposition}
\begin{proof}
Assume (i) and let $N:=M\otimes_{\mathbb{Z}[P]}\mathbb{Z}[Q]$. Applying \Cref{primeideals} (ii) with $K=\emptyset$ we get $N\cong M\otimes_{\mathbb{Z}[P]}^L\mathbb{Z}[Q]$. Thus, for any ideal $K\subset Q$ we have $N\otimes_{\mathbb{Z}[Q]}^L(\mathbb{Z}[Q]/K\mathbb{Z}[Q])\cong M\otimes_{\mathbb{Z}[P]}^L(\mathbb{Z}[Q]/K\mathbb{Z}[Q])$ which is quasi-isomorphic to $N/KN$ by \Cref{primeideals} (ii). Now (i) follows from \Cref{primeideals} (iii). For the second assertion, note that $Q$ is a finitely generated $P$-module (\cite[I, 4.3.6]{logbook}), so, by \Cref{nonpointedPflat}, the assertion follows from \cite[6.1.56]{gabberramero}.
\end{proof}

In the local noetherian setting, one can check monoidal flatness at a single prime ideal.

\begin{proposition}\label{maxideals}
Let $P$ be a fine monoid, $\alpha:P\to S$ a homomorphism to a noetherian local ring $S$ with maximal ideal $\mathfrak{m}$, and $\mathfrak{p}:=\alpha^{-1}\mathfrak{m}$. For a finitely generated $S$-module $M$, the following are equivalent.
\begin{enumerate}[label=\normalfont(\roman*)]
\item $M$ is $P$-flat.
\item The map \textup{\cref{idealtensormap}} is injective for $J=\mathfrak{p}$.
\end{enumerate}
\end{proposition}
\begin{proof}
By \Cref{nonpointedPflat}, this follows from \cite[6.4.31]{gabberramero}.
\end{proof}

\subsection{Dimension and depth formulae}
We show that $P$-flat local rings have particularly nice dimension and depth formulae, just like in the log regular case. We begin with dimension.

\begin{proposition}\label{diminequality}
Let $P$ be a sharp fine monoid, $P\to S$ a homomorphism to a noetherian local ring $S$ such that the image of $\mathfrak{p}:=P\setminus\{1\}$ lies in the maximal ideal of $S$. We have
\[ \dim S\leq \rk P^{\gp}+\dim S/\mathfrak{p}S \]
with equality if $S$ is $P$-flat.
\end{proposition}
\begin{proof}
Note that, since $P$ is sharp, our hypothesis on $P\to S$ implies it is a local homomorphism of monoids. Thus, by \cite[6.4.27]{gabberramero}, we have
\begin{equation}\label{dimineqeq}
\dim S\leq \dim P+\dim S/\mathfrak{p}S
\end{equation}
On the other hand, if $S$ is $P$-flat, then, by \cite[6.4.34]{gabberramero} (and \Cref{idealcorrespondence} and \Cref{nonpointedPflat}), given prime ideals $\mathfrak{Q}\subset S$ and $\mathfrak{r}\subset\mathfrak{q}:=\mathfrak{Q}\times_SP$ there exists a prime ideal $\mathfrak{R}\subset\mathfrak{Q}$ such that $\mathfrak{r}=\mathfrak{R}\times_SP$. Since $\mathfrak{p}$ is the maximal ideal of $P$, it is the inverse image of any minimal prime ideal of $S/\mathfrak{p}S$, so this implies
\[ \dim S\geq \dim P+\dim S/\mathfrak{p}S \]
hence we have equality in \cref{dimineqeq} if $S$ is $P$-flat. As $\dim P=\rk P^{\gp}$ (\cite[I, 1.4.7 (2)]{logbook}), the result follows.
\end{proof}

For depth we have a similar formula.

\begin{theorem}\label{depthequality}
Let $P$ be a sharp fine monoid, $P\to S$ a homomorphism to a noetherian local ring $S$ such that the image of $\mathfrak{p}:=P\setminus\{1\}$ lies in the maximal ideal of $S$. If $S$ is $P$-flat, then
\[ \depth(S)\leq\rk P^{\gp}+\depth(S/\mathfrak{p}S) \]
with equality if $P$ is saturated.
\end{theorem}

We first show a lemma.

\begin{lemma}\label{pdlemma}
With notation and assumptions as in \Cref{depthequality}. Assume $S$ is complete and let $R\to S$ be a finite local homomorphism with $R$ a regular local ring. If $S$ is $P$-flat, then $\Tor_j^{R[[P]]}(R,S)=0$ for $j>0$ and $S$ has finite projective dimension as an $R[[P]]$-module.
\end{lemma}
\begin{proof}
Let $k$ be the residue field of $R$. Consider the $\Tor$ spectral sequence for the composition of the quotient map $R[[P]]\to R$ of ideal generated by $P\setminus\{1\}$ with $R\twoheadrightarrow k$
\[ \Tor_i^R(k,\Tor_j^{R[[P]]}(R,S))\Rightarrow \Tor_{i+j}^{R[[P]]}(k,S) \]
Since $R$ is regular, we have $\Tor_i^R(k,-)=0$ for $i>\dim R$, so, by \cite[0\textsubscript{IV}, 17.2.6]{ega}, it suffices to show that $\Tor_j^{R[[P]]}(R,S)=0$ for $j>0$. To see this, note that $R[P]\to R[[P]]$ is flat (\cite[0\textsubscript{I}, 7.3.3]{ega}), hence $R\otimes_{R[P]}^LR[[P]]\cong R\otimes_{R[P]}R[[P]]\cong R$, whence $R\otimes_{R[[P]]}^LS\cong R\otimes_{R[P]}^LR[[P]]\otimes_{R[[P]]}^LS\cong R\otimes_{R[P]}^LS$. Thus,
\[ \Tor_j^{R[[P]]}(R,S)\cong \Tor_j^{R[P]}(R,S) \]
for all $j$. By \Cref{primeideals}, this implies $\Tor_j^{R[[P]]}(R,S)=0$ for $j>0$.
\end{proof}

\begin{proof}[Proof of \Cref{depthequality}]
We may assume $S$ is complete. By the Cohen structure theorem (\cite[0\textsubscript{IV}, 19.8.8]{ega}), there exists a complete regular local ring $R$ with a surjective (hence finite) local homomorphism $R\to S$. So, by \Cref{pdlemma}, $S$ has finite projective dimension as an $R[[P]]$-module and $\Tor_j^{R[[P]]}(R,S)=0$ for $j>0$. By a well-known result of Auslander (\cite[1.2]{auslander}), this implies
\[ \depth (R)=\depth (S\otimes_{R[[P]]}R)+\PD_{R[[P]]}(S) \]
On the other hand, by the Auslander--Buchsbaum formula, we have
\[ \PD_{R[[P]]}(S)+\depth (S)=\depth (R[[P]])\leq \dim (R[[P]])=\dim R+\rk P^{\gp} \]
where we have used \Cref{diminequality} for the last identity. Since $\dim R=\depth R$ and $S\otimes_{R[[P]]}R\cong S/\mathfrak{p}S$, the desired inequality follows. If $P$ is saturated, then $R[[P]]$ is Cohen--Macaulay (cf. \cite[I, 3.4.3]{logbook}) and we have equality.
\end{proof}

\subsection{Local presentation}
A complete log regular local ring has an explicit presentation as a quotient of a monoid algebra. Here we show that, in the case $P^{\gp}$ is torsion free, a $P$-flat local ring is flat over a similar quotient, under a certain assumption on its reduction modulo the maximal ideal of $P$. For example, it holds if the local ring contains a field.

\begin{proposition}\label{localpresentation}
Let $P$ be a sharp fine monoid such that $P^{\gp}$ is torsion free, and $P\to S$ a homomorphism to a noetherian local ring such that the image of $\mathfrak{p}:=P\setminus\{1\}$ lies in the maximal ideal of $S$. The following are equivalent.
\begin{enumerate}[label=\normalfont(\roman*)]
\item $S$ is $P$-flat and $\Spec(S/\mathfrak{p}S)\to\Spec(\mathbb{Z})$ is flat over its scheme-theoretic image.\footnote{\cite[6.10.1]{ega1}}
\item $\Spec(S)\to\Spec(\mathbb{Z}[P])$ is flat over its scheme-theoretic image.
\item $S$ is flat over $\mathbb{Z}[P]/(\theta)$, for some $\theta\in\mathbb{Z}[P]$ whose constant term is zero if and only if $\theta=0$.
\end{enumerate}
\end{proposition}

We first show a couple of lemmata.

\begin{lemma}\label{imageflat}
Let $S$ be a noetherian local ring with maximal ideal $\mathfrak{m}$, $P$ be a fine monoid, $\alpha:P\to S$ a homomorphism, and $\mathfrak{p}:=\alpha^{-1}\mathfrak{m}$. Consider the following conditions.
\begin{enumerate}[label=\normalfont(\roman*)]
\item $\Spec(S)\to\Spec(\mathbb{Z}[P])$ is flat over its scheme-theoretic image.
\item $\Spec(S/\mathfrak{p}S)\to\Spec(\mathbb{Z}[P]/(\mathfrak{p}))$ is flat over its scheme-theoretic image.
\end{enumerate}
We have \textup{(i)} $\Rightarrow$ \textup{(ii)}. If $S$ is $P$-flat, then \textup{(ii)} $\Rightarrow$ \textup{(i)}.
\end{lemma}
\begin{proof}
Let $R=\mathbb{Z}[P]$ and $\bar{R}$ the image of $R$ in $S$; $\Spec(\bar{R})$ is the scheme-theoretic image of $\Spec(S)$. If (i) holds, then $\bar{R}\to S$ is flat, hence so is $\bar{R}/\mathfrak{p}\bar{R}\to S/\mathfrak{p}S$, whence (ii).

Now assume $S$ is $P$-flat.

\begin{claim}\label{firstclaim}
If $K\subset\ker(R\to S)$ is an ideal and $\tilde{R}:=R/K$, then $\Tor_1^{\tilde{R}}(\tilde{R}/\mathfrak{p}\tilde{R},S)=0$.
\end{claim}
\begin{proof}[Proof of claim]\renewcommand{\qedsymbol}{$\Diamond$}
By $P$-flatness, the map $(\mathfrak{p})\otimes_RS\to S$ is injective and factors over a surjective map $(\mathfrak{p})\otimes_RS\twoheadrightarrow \mathfrak{p}\tilde{R}\otimes_{\tilde{R}}S$. It follows that the map $\mathfrak{p}\tilde{R}\otimes_{\tilde{R}}S\to S$ is also injective, which implies the claim.
\end{proof}

\noindent In particular, $\Tor_1^{\bar{R}}(\bar{R}/\mathfrak{p}\bar{R},S)=0$ and (ii) implies (i) by the local flatness criterion (\cite[0\textsubscript{III}, 10.2.2]{ega}).
\end{proof}

\begin{lemma}\label{depthlemma2}
Let $S$ be a noetherian local ring with maximal ideal $\mathfrak{m}$, $P$ be a fine monoid, $\alpha:P\to S$ a homomorphism, and $\mathfrak{p}:=\alpha^{-1}\mathfrak{m}$. For a finitely generated $S$-module $M$ and an $M$-regular sequence $x_1,...,x_n\in S$, the following are equivalent.
\begin{enumerate}[label=\normalfont(\roman*)]
\item $M/(x_1,...,x_m)M$ is $P$-flat for $0\leq m\leq n$.
\item The images of $x_1,...,x_n$ in $M/\mathfrak{p}M$ form a regular sequence.
\end{enumerate}
\end{lemma}
\begin{proof}
For a finitely generated $S$-module $N$ and $N$-regular element $x\in S$, we have a commutative diagram with exact rows
\[ \begin{tikzcd}
& (\mathfrak{p})\otimes_{\mathbb{Z}[P]}N \arrow[r,"x"] \arrow[d] & (\mathfrak{p})\otimes_{\mathbb{Z}[P]}N \arrow[r] \arrow[d] & (\mathfrak{p})\otimes_{\mathbb{Z}[P]}(N/xN) \arrow[r] \arrow[d] & 0 \\
0 \arrow[r] & N \arrow[r,"x"] & N \arrow[r] & N/xN & 
\end{tikzcd} \]
and the snake lemma combined with \Cref{maxideals} proves the equivalence of (i) and (ii) in the case $n=1$. We prove the general case by induction on $n$. If (i) holds, then, as $x_n$ is regular on $N:=M/(x_1,...,x_{n-1})$, the diagram proves that $x_n$ is regular on $N/\mathfrak{p}N$, so by induction (ii) holds. Conversely, if (ii) holds, then $x_n$ is regular on $M/(x_1,...,x_{n-1},\mathfrak{p})M=N/\mathfrak{p}N$, so the commutative diagram and induction imply (i).
\end{proof}

\begin{proof}[Proof of \Cref{localpresentation}]
By \Cref{imageflat}, (i) and (ii) are equivalent, so it suffices to show that (iii) $\Rightarrow$ (i) and (ii) $\Rightarrow$ (iii).

Let $R=\mathbb{Z}[P]$; $R$ is an integral domain since $P^{\gp}$ is torsion free and contains $P$, hence any nonzero element of $R$ is regular. In particular, for $\theta$ as in (iii), $\mathbb{Z}[P]/(\theta)$ is $P$-flat by \Cref{depthlemma2}. Thus, (iii) implies (i) and it remains to show that (ii) implies (iii). For this we may assume $S=R/I$ is a quotient of $R$.

Let $n$ be a generator of the kernel of $\mathbb{Z}\cong R/\mathfrak{p}R\to S/\mathfrak{p}S$.  We have $n\in\mathfrak{p}S$, so we can write $n=\sum_{p\in\mathfrak{p}}s_p[p]$ for some $s_p\in S$. Choose lifts $r_p\in R$ of $s_p$ for all $p$ and define $\theta:=n-\sum_{p\in\mathfrak{p}}r_p[p]$; then $\theta\in I$.

Now, by \Cref{firstclaim}, the $P$-flatness of $S$ implies $\Tor_1^{R/(\theta)}(R/(\theta,\mathfrak{p}),S)=0$. Since $S/\mathfrak{p}S$ is flat over $R/(\theta,\mathfrak{p})\cong\mathbb{Z}/n$ (by \Cref{imageflat}), the local flatness criterion (\cite[0\textsubscript{III}, 10.2.2]{ega}) implies that $S$ is flat over $R/(\theta)$. So it remains to show that $\theta$ has nonzero constant term if $\theta\neq 0$. For this we may assume $R=\mathbb{Z}_l[[P]]$, where $l$ is the residue characteristic of $S$, and $S=R/(\theta)$. Then, by \Cref{depthlemma2}, we have $\theta\not\in\mathfrak{p}R$, as required.
\end{proof}

\section{t-flatness}\label{tflatsection}
Here we extend the results of the previous section to fine log schemes using the notion of t-flatness for log schemes (\Cref{tflatdef}), an intrinsic property that is local for the log flat topology. We introduce the notion of logarithmic dimension (resp. logarithmic depth) and show that it coincides with the classical notion for t-flat fine (resp. fs) log schemes. We also discuss the log cotangent complex of t-flat schemes and use it to give homological characterisations of t-flatness and related local properties.

For a log scheme $(X,M_X)$ and geometric point $x\to X$,\footnote{see \cref{conventions} for our conventions on geometric points} define $C(M_{X,x})\subset X_{(x)}$ to be the closed subscheme of ideal generated by the image of $M_{X,x}\setminus M_{X,x}^*$ in $\mathcal{O}_{X,x}$. We call $C(M_{X,x})$ the \emph{logarithmic stratum \emph{(of $(X_{(x)},M_X|_{X_{(x)}})$)} through $x$}. For a fine Zariski log structure $M_X$ and $x\in X$, one can define $C(M_{X,x})$ as a closed subscheme of the localisation of $X$ at $x$; then, as $x$ varies, the union of the $C(M_{X,x})$ forms a stratification of $X$, called the \emph{logarithmic \emph{(or \emph{canonical})} stratification} (\cite[2.2.10]{ATtorification}).

We assume all schemes are locally noetherian.

\subsection{Definition and basic properties}
We begin with a definition.

\begin{definition}\label{tflatdef}
Let $(X,M_X)$ be a log scheme with integral log structure and $x\to X$ a geometric point. We say that $(X,M_X)$ is \emph{t-flat at $x$} if $\mathcal{O}_{X,x}$ is $M_{X,x}$-flat.

We say that $(X,M_X)$ is \emph{t-flat} if it is $t$-flat at every geometric point $x\to X$.
\end{definition}

\begin{remark}
This notion is due to Gabber (cf. \cite{toricflattening}, \cite{thompson}), who defined it more generally for $\mathcal{O}_X$-modules to show a toric analogue of the classical Raynaud--Gruson flattening theorem. The terminology \emph{t-flat} (\emph{t} for \emph{toric}) is natural since an fs log scheme over a field is t-flat if and only if, locally for the \'etale topology, it is flat over a toric variety (cf. \Cref{tlflogflat}).
\end{remark}

\begin{example}
A log regular scheme is t-flat (cf. \cite[6.1]{logreg}). In particular, if $P$ is an fs monoid, then $\D(P)$ is t-flat.\footnote{see \cref{conventions} for the notation $\D(P)$} In fact, this is true for any fine monoid $P$ and, for a scheme $S$, $\D(P)\times S$ is also t-flat (see \Cref{tlflogflat} below).
\end{example}

\begin{proposition}\label{logstrinjective}
Let $(X,M_X)$ be a log scheme with integral log structure. If $(X,M_X)$ is t-flat at a geometric point $x\to X$, then the structure map $M_{X,x}\to\mathcal{O}_{X,x}$ is injective.
\end{proposition}
\begin{proof}
We may assume $X$ is strictly local of closed point $x$. Write $M:=M_{X,x}$ and let $\mathcal{M}$ denote the log structure associated to the canonical map $M\to\mathcal{O}_{X,x}\cong\Gamma(X,\mathcal{O}_X)$. Then $\mathcal{M}_x=M$ and, since $\mathcal{O}_{X,x}$ is $M$-flat, the proof of \cite[12.5.2]{gabberramero} shows that the map $\mathcal{M}_x\to\mathcal{O}_{X,x}$ is injective.
\end{proof}

The following key property, which depends on local noetherianity, is due to Gabber \cite{toricflattening}.

\begin{proposition}\label{chartindependent}
Assume $(X,M_X)$ is fine. The following are equivalent.
\begin{enumerate}[label=\normalfont(\roman*)]
\item $(X,M_X)$ is t-flat at a geometric point $x\to X$.
\item For any fine chart $P\to M_{X}$, $\mathcal{O}_{X,x}$ is $P$-flat.
\end{enumerate}
\end{proposition}
\begin{proof}
By \cite[12.1.11]{gabberramero} and \Cref{nonpointedPflat}, we may pass to pointed monoids and then the assertion follows from \cite[12.5.1]{gabberramero}.
\end{proof}

In particular, this implies that t-flatness is open in the \'etale topology.

\begin{corollary}\label{gentflat}
Assume $(X,M_X)$ is fine. If $(X,M_X)$ is t-flat at a geometric point $x\to X$, then $(X,M_X)$ is t-flat in an \'etale neighbourhood of $x$.
\end{corollary}
\begin{proof}
We may assume we have a fine chart $P\to M_{X}$. Then, for a prime ideal $J\subset P$, the map $(J)\otimes_{\mathbb{Z}[P]}\mathcal{O}_X\to \mathcal{O}_X$ is injective at $x$, hence also in an open neighbourhood $U_J$. Since $\Spec(P)$ is finite (\cite[5.5]{logreg}), setting $U:=\cap_{J\in\Spec(P)}U_J$ we deduce from \Cref{primeideals} that $\mathcal{O}_{X,y}$ is $P$-flat for any geometric point $y\to U$. By \Cref{chartindependent}, this implies that $(U,M_X|_U)$ is t-flat.
\end{proof}

\begin{corollary}\label{flatcover}
Let $f:(X,M_X)\to (Y,M_Y)$ be a morphism of fine log schemes and $x\to X$ a geometric point. If $f$ is strict log flat at $x$, then $(X,M_X)$ is t-flat at $x$ if and only if $(Y,M_Y)$ is t-flat at $f(x)$.
\end{corollary}
\begin{proof}
Up to localising, by \cite[IV, 4.1.2 (2)]{logbook} we may assume $X\to Y$ is faithfully flat and then the result follows from \Cref{chartindependent}.
\end{proof}

\begin{corollary}\label{tlfcomplete}
Let $(X,M_X)$ be a fine log scheme and $\hat{X}_{(x)}$ the completion of the strict localisation $X_{(x)}$ of $X$ at a geometric point $x\to X$. The following are equivalent.
\begin{enumerate}[label=\normalfont(\roman*)]
\item $(X,M_X)$ is t-flat at $x$
\item $(\hat{X}_{(x)},M_{X}|_{\hat{X}_{(x)}})$ is t-flat at $x$.
\end{enumerate}
\end{corollary}
\begin{proof}
Since $X_{(x)}$ is noetherian, $\hat{X}_{(x)}\to X$ is flat and \Cref{flatcover} applies.
\end{proof}

\subsection{Relation to log flatness}
Over a point with trivial log structure, a (locally noetherian) fine log scheme is t-flat if and only if it is log flat.

\begin{proposition}\label{tlflogflat}
Let $(X,M_X)$ be a fine log scheme, $X\to S$ a morphism of schemes and $x\to X$ be a geometric point. View $S$ as a log scheme with trivial log structure.
\begin{enumerate}[label=\normalfont(\roman*)]
\item If $(X,M_X)\to S$ is log flat at $x$, then $(X,M_X)$ is t-flat at $x$.
\item If $S$ is the spectrum of a field and $(X,M_X)$ is t-flat at $x$, then $(X,M_X)\to S$ is log flat at $x$.
\end{enumerate}
\end{proposition}
\begin{proof}
If $(X,M_X)\to S$ is log flat, then locally for the fppf topology there is a fine chart $P\to M_X$ and a strict log flat map $(X,M_X)\to \D(P)\times S$. We claim that $\D(P)\times S$ is t-flat at the image of $x$. Indeed, this is local so we may assume $S=\Spec(R)$ and then, by \Cref{chartindependent}, it suffices to show that $R[P]$ is $P$-flat, which follows from \Cref{primeideals} (iii). Now (i) follows from \Cref{flatcover}. For (ii), let $S=\Spec(k)$. By \cite[III, 1.2.7 (4)]{logbook} (or \Cref{neatcharts} below), after passing to a flat cover of $X$, we may assume there is a neat chart $P\to M_{X}$ with $P$ sharp fine. By t-flatness, we have $\Tor_1^{k[P]}(k,\mathcal{O}_{X,x})=0$, which implies that $X\to\Spec(k[P])$ is flat at $x$ (\cite[0\textsubscript{III}, 10.2.2]{ega}), hence (ii).
\end{proof}

In the last proof we used the existence of \emph{neat} charts in the sense of \cite[II, \S 2]{logbook}. The following lemma gathers the results on charts we will need.

\begin{lemma}\label{neatcharts}
Let $f:(X,M_X)\to (Y,M_Y)$ be a morphism of fine log schemes and $x\to X$ a geometric point.
\begin{enumerate}[label=\normalfont(\roman*)]
\item Locally for the \'etale topology at $f(x)$, there is a fine chart $Q\to M_Y$ that is exact\footnote{\cite[II, 2.3.1]{logbook}} at $f(x)$.
\item Locally for the fppf topology at $f(x)$, there is a fine chart $Q\to M_Y$ that is neat\footnote{\cite[II, 2.3.1]{logbook}} at $f(x)$. If $\overline{M}_{Y,f(x)}^{\gp}$ is torsion free, then this chart exists locally for the \'etale topology.
\item Assume given a fine chart $Q\to M_Y$. Locally for the \'etale topology at $x$, there exists an injective chart $Q\to P$ for $f$ with $P$ fine and $P\to M_X$ exact at $x$. If the log structures are fs and $Q^{\gp}$ is torsion free, we may take $P$ such that $P^{\gp}$ is torsion free.
\item Assume given a fine chart $Q\to M_Y$. Locally for the fppf topology at $x$, there is a neat\footnote{\cite[II, 2.4.4]{logbook}} chart $Q\to P$ for $f$ with $P\to M_X$ a fine exact chart at $x$.
\item Assume given a chart as in \textup{(ii)}. If $f$ is s-injective\footnote{in the sense of \cite{logbook}, i.e., $\overline{M}_{Y,f(x)}\to\overline{M}_{X,x}$ is injective} at $x$, then, in \textup{(iv)}, we can take $P\to M_X$ neat at $x$. If $f$ is s-finite\footnote{in the sense of \cite[III, 2.3.2]{logbook}, i.e., $\overline{M}_{X,x}$ is a finitely generated $\overline{M}_{Y,f(x)}$-module} at $x$ and $\ker(\overline{M}_{Y,f(x)}^{\gp}\to \overline{M}_{X,x}^{\gp})$ is torsion, then, in \textup{(iv)}, we can take $Q\to P$ to be a Kummer\footnote{recall that we do not assume \emph{Kummer} to necessarily imply \emph{saturated}} homomorphism. In particular, if $f$ is Kummer at $x$, then, in \textup{(iv)}, we can take $P\to M_X$ to be neat at $x$ and $Q\to P$ Kummer.
\item Assume given a neat chart $Q\to P$ of fine monoids for $f$ at $x$. If $f$ is log flat, then $(X,M_X)\to (Y,M_Y)\times_{\D(Q)}\D(P)$ is strict log flat at $x$.
\item If $f$ is log \'etale, then \textup{(iv)} holds locally for the \'etale topology at $x$, the order of $P^{\gp}/Q^{\gp}$ is invertible at $x$, and $(X,M_X)\to (Y,M_Y)\times_{\D(Q)}\D(P)$ is strict log \'etale.
\end{enumerate}
\end{lemma}
\begin{proof}
(i) is \cite[II, 2.3.5]{logbook} and (ii) is \cite[2.1]{olsson} (or \cite[III, 1.2.7 (2)]{logbook}). For (iii) use \cite[2.2]{olsson} and the fact that any fine chart has a localisation that is a fine exact chart \cite[II, 2.3.2]{logbook}. (iv) is \cite[III, 1.2.7 (4)]{logbook}. Since $P\to M_{X,x}$ is exact, we have an isomorphism $\overline{P}\cong\overline{M}_{X,x}$ (\cite[II, 2.3.1]{logbook}), hence $P^{\gp}/P^*\cong\overline{M}_{X,x}^{\gp}$. For (v), from the commutative diagram with exact rows
\begin{equation}
\begin{tikzcd}
0 \ar[r] & Q^{\gp} \ar[d,"\sim" labl] \ar[r] & P^{\gp} \ar[r] \ar[d] & P^{\gp}/Q^{\gp} \ar[d,"\sim" labl] \ar[r] & 0 \\
& \overline{M}_{Y,f(x)}^{\gp} \ar[r] & \overline{M}_{X,x}^{\gp} \ar[r] & \overline{M}_{X,x}^{\gp}/\overline{M}_{Y,f(x)}^{\gp} \ar[r] & 0
\end{tikzcd}\label{commdiag}
\end{equation}
we deduce that $P^*\cong\ker(\overline{M}_{Y,f(x)}^{\gp}\to \overline{M}_{X,x}^{\gp})$. Thus, if $f$ is s-injective at $x$, $P^*=\{1\}$ and $P\to M_X$ is neat at $x$. If $\ker(\overline{M}_{Y,f(x)}^{\gp}\to \overline{M}_{X,x}^{\gp})$ is torsion, then $P^*$ is torsion, hence finite. So, if $f$ is s-finite at $x$, $Q\to\overline{M}_{X,x}\cong P/P^*$ is $\mathbb{Q}$-surjective (\cite[I, 4.3.6]{logbook}), hence so is $Q\to P$. Thus, $Q\to P$ is a Kummer homomorphism. This proves (v) and (vi) can be proven as \cite[IV, 4.1.7]{logbook} (the local finite presentation assumption is superfluous since our schemes are locally noetherian, cf. \cite[IV\textsubscript{3}, 11.3.10]{ega}).

Finally, if $f$ is log \'etale, then, by \cite[IV, 3.3.1 (3)]{logbook}, the chart of (iv) exists \'etale locally on $X$ and the order of $P^{\gp}/Q^{\gp}$ is invertible at $x$. We have a strict morphism $(X,M_X)\to (Y,M_Y)\times_{\D(Q)}\D(P)$ of log \'etale fine $(Y,M_Y)$-log schemes, which must itself be log \'etale (use the infinitesimal criterion \cite[3.3]{log}), proving (vi).
\end{proof}

Since an fppf covering can, Zariski-locally on the target, be refined to a quasi-finite covering (\cite[IV\textsubscript{4}, 17.16.2]{ega}), \Cref{neatcharts} (v)-(vi) implies the following.

\begin{corollary}\label{goodchart}
Let $f:(X,M_X)\to (Y,M_Y)$ be a Kummer, log flat morphism of fine log schemes and $x\to X$ a geometric point. Up to passing to quasi-finite flat neighbourhoods of $x$ and $f(x)$, there are charts $P\to M_X$ and $Q\to M_Y$ neat at $x$ and $f(x)$, respectively, and a Kummer homomorphism $Q\to P$ providing a neat chart for $f$ at $x$, such that the morphism $(X,M_X)\to (Y,M_Y)\times_{\D(Q)}\D(P)$ is strict log flat at $x$.
\end{corollary}

Next, we show that t-flatness descends under (not necessarily fs) Kummer log flat morphisms; in particular, it is local for the log flat topology.

\begin{proposition}\label{klftlf}
Let $f:(X,M_X)\to (Y,M_Y)$ be a log flat morphism of log schemes and $x\to X$ a geometric point. If $(Y,M_Y)$ is t-flat at $f(x)$, then $(X,M_X)$ is t-flat at $x$. The converse holds if, locally for the fppf topology, there is a neat Kummer chart $Q\to P$ of fine monoids for $f$ at $x$.
\end{proposition}
\begin{proof}
By \Cref{neatcharts} (vi), locally for the fppf topology there is a strict log flat morphism $(X,M_Y)\to (Y,M_Y)\times_{\D(Q)}\D(P)$. Thus, by \Cref{flatcover}, we are reduced to the case $(X,M_X)=(Y,M_Y)\times_{\D(Q)}\D(P)$. In this case, the result follows from \Cref{monoidflatness} and \Cref{chartindependent}.
\end{proof}

\begin{corollary}\label{saturationtflat}
Let $f:(X,M_X)\to (Y,M_Y)$ be the saturation of a fine log scheme and $x\to X$ a geometric point. $(Y,M_Y)$ is t-flat at $f(x)$ if and only if $(X,M_X)$ is t-flat at $x$.
\end{corollary}

We will need a lemma.

\begin{lemma}\label{saturationtorsion}
Let $f:(X,M_X)\to (Y,M_Y)$ be the saturation of a fine log scheme and $x\to X$ a geometric point.
\begin{enumerate}[label=\normalfont(\roman*)]
\item $f$ is s-finite.
\item If $Q\to M_Y$ is a chart that is exact at $f(x)$, then $Q^{\sat}\to M_X$ is an exact chart at $x$ and $Q\to Q^{\sat}$ is a neat Kummer chart for $f$ at $x$. Moreover, if $Q$ is fine, then so is $Q^{\sat}$.
\item $\ker(\overline{M}_{Y,f(x)}^{\gp}\to\overline{M}_{X,x}^{\gp})$ is torsion.
\end{enumerate}
\end{lemma}
\begin{proof}
This is local so we may assume $(X,M_X)= (Y,M_Y)\times_{\D(Q)}\D(Q^{\sat})$ for a chart $Q\to M_Y$ (cf. \cite[12.2.37]{gabberramero}), that we can assume to be exact at $f(x)$ (\Cref{neatcharts} (i)). Then, by \cite[II, 1.1.8 (1)]{logbook}, we have $Q^{\sat}/\alpha^{-1}(\mathcal{O}_{X,x}^*)\cong\overline{M}_{X,x}$, where $\alpha:Q^{\sat}\to\mathcal{O}_{X,x}$ is the canonical homomorphism. If $q\in \alpha^{-1}(\mathcal{O}_{X,x}^*)\subset Q^{\sat}$, then there is an integer $n\geq 1$ such that $q^n\in Q$. Since $Q\to M_{Y,f(x)}$ is local (it is exact), so is $Q\to\mathcal{O}_{X,x}$, hence $q^n\in Q^*$. Thus, $\alpha^{-1}(\mathcal{O}_{X,x}^*)\subset (Q^{\sat})^*$ and, since the reverse inclusion trivially holds, we have $\overline{Q^{\sat}}\cong \overline{M}_{X,x}$ and $Q^{\sat}\to M_{X,x}$ is exact by \cite[I, 4.2.1 (3)]{logbook}. By \cite[I, 1.3.5]{logbook}, this implies that $Q^{\gp}=(Q^{\sat})^{\gp}\to\overline{M}_{X,x}^{\gp}$ is surjective, hence so is $\overline{M}_{Y,f(x)}^{\gp}\to\overline{M}_{X,x}^{\gp}$, thus $Q\to Q^{\sat}$ is neat. Since it is trivially Kummer, this proves the first assertion of (ii); the second is \cite[6.4.1 (ii)]{gabberramero}.

For (i), by \cite[I, 4.3.6]{logbook}, it suffices to show that $\overline{M}_{Y,f(x)}\to\overline{M}_{X,x}$ is $\mathbb{Q}$-surjective, which, by (ii), follows from the analogous fact for $Q\to Q^{\sat}$. For (iii), up to localising for the fppf topology at $f(x)$, we may assume $Q\cong\overline{M}_{Y,f(x)}$. Then, by (ii), the diagram \cref{commdiag} shows that $\ker(\overline{M}_{Y,f(x)}^{\gp}\to\overline{M}_{X,x}^{\gp})\cong (Q^{\sat})^*$. Now (iii) follows from the fact that the group of units of the saturation of a sharp integral monoid is torsion (\cite[I, 1.3.5 (5)]{logbook}).
\end{proof}

\begin{proof}[Proof of \Cref{saturationtflat}]
By \Cref{klftlf}, this follows from \Cref{neatcharts} (i) and  \Cref{saturationtorsion} (ii).
\end{proof}

In the presence of t-flatness, we now relate log flatness to the flatness of the logarithmic strata.

\begin{theorem}\label{goodcharts}
Let $f:(X,M_X)\to (Y,M_Y)$ be a morphism of fine log schemes. Assume that at a geometric point $x\to X$ there exists an injective chart $Q\to P$ of fine monoids for $f$, and consider the following.
\begin{enumerate}[label=\normalfont(\roman*)]
\item $(X,M_X)\to (Y,M_Y)\times_{\D(Q)}\D(P)$ is strict log flat at $x$.
\item $(X,M_X)\to (Y,M_Y)$ is log flat at $x$.
\item $C(M_{X,x})\to C(M_{Y,f(x)})$ is flat.
\end{enumerate}
We have \textup{(i)} $\Rightarrow$ \textup{(ii)} $\Rightarrow$ \textup{(iii)}. If $(X,M_X)$ is t-flat at $x$ and $P\to M_X$ (resp. $Q\to M_Y$) is neat at $x$ (resp. $f(x)$), then all conditions are equivalent.
\end{theorem}

Note that the implication \textup{(i)} $\Rightarrow$ \textup{(ii)} follows from the definition of log flatness (\cite[1.10]{logflat}). We first show that condition (iii) makes sense.

\begin{lemma}\label{logstratacharts}
Let $f:(X,M_X)\to (Y,M_Y)$ be a morphism of log schemes and $x\to X$ a geometric point.
\begin{enumerate}[label=\normalfont(\roman*)]
\item A local homomorphism of monoids $P\to M_{X,x}$ induces a closed immersion $C(M_{X,x})\to X_{(x)}\times_{\Spec(\mathbb{Z}[P])}\Spec(\mathbb{Z}[P^*])$. It is an isomorphism if, in addition, $P\to M_{X,x}$ induces a chart at $x$.
\item $f$ induces a morphism $C(M_{X,x})\to C(M_{Y,f(x)})$ such that $C(M_{X,x})\to C(M_{Y,f(x)})\times_{Y_{(f(x))}}X_{(x)}$ is a closed immersion (resp. an isomorphism if $f$ is strict).
\item Assume $(X,M_X)\cong (Y,M_Y)\times_{\D(Q)}\D(P)$ for a chart $Q\to P$ of $f$ at $x$. If $Q\to P$ and $P\to M_{X,x}$ are local homomorphisms, then $C(M_{X,x})$ is a strict localisation of $C(M_{Y,f(x)})\times_{\D(Q^*)}\D(P^*)$.
\item Assume $(X,M_X)$ is fine. For any geometric point $y\to C(M_{X,x})$, we have $\overline{M}_{X,y}\cong\overline{M}_{X,x}$ and $C(M_{X,x})_{(y)}=C(M_{X,y})$.
\end{enumerate}
\end{lemma}
\begin{proof}
For (i), since $P\to M_{X,x}$ is local it follows that the image of $P\setminus P^*$ lies in $M_{X,x}\setminus M_{X,x}^*$, hence the first assertion. If $P\to M_{X,x}$ induces a chart at $x$, then $\overline{P}\cong\overline{M}_{X,x}$ (cf. \cite[II, 1.1.8 (1)]{logbook}), which implies the second assertion of (i). Now (ii) follows by applying (i) with $P=M_{Y,f(x)}$.

For (iii), we may assume $Y$ is strictly local. By (i), we have $\mathcal{O}(C(M_{X,x}))\cong\mathcal{O}(X_{(x)})\otimes_{\mathbb{Z}[P]}\mathbb{Z}[P^*]$. Since $\mathcal{O}(X)\otimes_{\mathbb{Z}[P]}\mathbb{Z}[P^*]\cong\mathcal{O}(Y)\otimes_{\mathbb{Z}[Q]}\mathbb{Z}[P^*]$, $\mathcal{O}(C(M_{X,x}))$ is a strict localisation of the latter. Now, $Q\to P$ is local, so the image of $Q\setminus Q^*$ lies in $P\setminus P^*$ and $\mathbb{Z}[Q]\to\mathbb{Z}[P]\to\mathbb{Z}[P^*]$ factors over $\mathbb{Z}[Q^*]$. Hence,
\[ \begin{split}
\mathcal{O}(Y)\otimes_{\mathbb{Z}[Q]}\mathbb{Z}[P^*] &\cong\mathcal{O}(Y)\otimes_{\mathbb{Z}[Q]}\mathbb{Z}[Q^*]\otimes_{\mathbb{Z}[Q^*]}\mathbb{Z}[P^*] \\
&\cong\mathcal{O}(C(M_{Y,f(x)}))\otimes_{\mathbb{Z}[Q^*]}\mathbb{Z}[P^*]
\end{split} \]
which proves (iii).

For (iv), we may assume $X=X_{(x)}$ and there is an exact fine chart $P\to M_{X,x}$ (\Cref{neatcharts} (i)). Then, the log structure $M_X$ is Zariski (\cite[III, 1.4.1 (2)]{logbook}) and it suffices to show the assertion for a point $y\in C(M_{X,x})$. By \cite[II.1.4]{tsuji}, $\overline{P_{\mathfrak{p}}}\cong\overline{M}_{X,y}$, where $\mathfrak{p}\subset P$ is the (prime) inverse image of the maximal ideal of $\mathcal{O}_{X,y}$. Since $P\to M_{X,x}$ is exact, $\overline{P}\cong\overline{M}_{X,x}$ and $C(M_{X,x})\subset X$ is defined by the ideal generated by the image of $P\setminus P^*$ in $\mathcal{O}_{X,x}$. The condition $y\in C(M_{X,x})$ implies that the image of $P\setminus P^*$ in $\mathcal{O}_{X,y}$ is contained in the maximal ideal. Thus, $P\setminus P^*\subset\mathfrak{p}$, hence $P\setminus P^*=\mathfrak{p}$, whence $\overline{M}_{X,y}\cong\overline{M}_{X,x}$. Since, by (i), $C(M_{X,y})$ is the closed subscheme of $X_{(y)}$ defined by the ideal generated by the image of $\mathfrak{p}$, this implies (iv).
\end{proof}

Next, we show \textup{(ii)} $\Rightarrow$ \textup{(iii)} in \Cref{goodcharts}. Recall (\cref{conventions}) that, given a geometric point $x\to X$, we denote by $k(x)$ the residue field of $\mathcal{O}_{X,x}$.

\begin{proposition}\label{flatstrata}
Let $f:(X,M_X)\to (Y,M_Y)$ be a morphism of fine log schemes, $x\to X$ a geometric point, and $K:=\ker(\overline{M}_{Y,f(x)}^{\gp}\to\overline{M}_{X,x}^{\gp})$.
\begin{enumerate}[label=\normalfont(\roman*)]
\item If $f$ is log flat, then $C(M_{X,x})\to C(M_{Y,f(x)})$ is flat.
\item If $f$ is log \'etale, then $C(M_{X,x})\to C(M_{Y,f(x)})$ is a strict localisation of a finite type Cohen--Macaulay morphism of relative dimension $\rk K$, and setting $\varepsilon(x):=\rk K-\td (k(x)/k(f(x)))$, we have $\varepsilon(x)\geq 0$ and
\begin{equation}\label{stratadimdepth}
\delta (C(M_{X,x}))=\delta (C(M_{Y,f(x)}))+\varepsilon(x)
\end{equation}
for $\delta=\dim$ (resp. $\delta=\depth$).
\item If $f$ is log \'etale and $K$ is torsion free, then $C(M_{X,x})\to C(M_{Y,f(x)})$ is a strict localisation of a smooth morphism of relative dimension $\rk K$.
\item If $f$ is log \'etale and $K=0$, then $C(M_{X,x})\cong C(M_{Y,f(x)})$.
\end{enumerate}
\end{proposition}
\begin{proof}
By \Cref{logstratacharts} (ii), this is local for the fpqc topology on $Y$. By \Cref{neatcharts} (ii), we may therefore assume $M_Y$ admits a fine chart $Q\to M_Y$ that is neat at $f(x)$.

Suppose $f$ is log flat (resp. log \'etale). Then, by \Cref{neatcharts}, there is a flat quasi-finite (resp. \'etale) morphism $X'\to X$, a geometric point $x'\to X'$ lying above $x$, and a neat chart $Q\to P$ for $f':(X',M_{X'})\to (Y,M_Y)$, where $M_{X'}=M_X|_{X'}$, with $P\to M_{X',x'}$ exact, such that we have a commutative diagram of fine log schemes
\[ \begin{tikzcd}
(X',M_{X'}) \arrow[d] \arrow[r] & (Y,M_Y)\times_{\D(Q)}\D(P) \arrow[d] \\
(X,M_X) \arrow[r] & (Y,M_Y)
\end{tikzcd} \]
where the top horizontal map is strict log flat (resp. strict log \'etale). Note that $Q\to \mathcal{O}_{Y,f(x)}$ is local (since $Q\to M_{Y,f(x)}$ is neat), hence so is the composition $Q\to\mathcal{O}_{Y,f(x)}\to\mathcal{O}_{X',x'}$. Thus, $Q\to P$ is local and, since $P\to M_{X',x'}$ is as well (it is exact), it follows from \Cref{logstratacharts} that we have a commutative diagram of schemes
\begin{equation}\label{stratadiagram}
\begin{tikzcd}
C(M_{X',x'}) \arrow[d] \arrow[r] & C(M_{Y,f(x)})\times\D(P^*) \arrow[d] \\
C(M_{X,x}) \arrow[r] & C(M_{Y,f(x)})
\end{tikzcd}
\end{equation}
where the left vertical map is flat (resp. an isomorphism) and the top horizontal map is flat (resp. a strict localisation). Thus, the composition $C(M_{X',x'})\to C(M_{X,x}) \to C(M_{Y,f(x)})$ is flat, which, since $C(M_{X',x'})\to C(M_{X,x})$ is faithfully flat (\Cref{logstratacharts} (ii)), implies (i).

For (ii), note that $P^*\cong K$ since $Q\to P$ is neat (cf. \cref{commdiag}). In particular, $P^*$ is a finitely generated abelian group, so \cref{stratadiagram} implies the first assertion of (ii). If $K$ is torsion free, then $P^*$ is free, hence (iii), and (iv) follows from (iii). For the equation \cref{stratadimdepth}, by \cite[IV\textsubscript{2}, 6.1.2]{ega} (resp. \cite[IV\textsubscript{2}, 6.3.1]{ega}) it suffices to show $\dim C(M_{X,x})_y=\varepsilon(x)$, where $C(M_{X,x})_y$ denotes the closed fibre over $C(M_{Y,y})$. Since $\Spec(k(y)[P^*])$ is equidimensional of dimension $\rk K$, this follows from \cite[IV\textsubscript{2}, 5.2.3]{ega}.
\end{proof}

\begin{proof}[Proof of \Cref{goodcharts}]
It remains to show the final assertion. Since $(X,M_X)$ is t-flat at $x$, it follows from the assumption on $P$ and $Q$ and \cite[0\textsubscript{III}, 10.2.2]{ega} that $(X,M_X)\to (Y,M_Y)\times_{\D(Q)}\D(P)$ is strict log flat if and only if $C(M_{X,x})\to C(M_{Y,f(x)})$ is flat.
\end{proof}

\begin{corollary}\label{Kummerflatstrata}
Let $f:(X,M_X)\to (Y,M_Y)$ be a morphism of fine log schemes and $x\to X$ a geometric point. If $f$ is s-injective at $x$ and $(X,M_X)$ is t-flat at $x$, then the following are equivalent.
\begin{enumerate}[label=\normalfont(\roman*)]
\item $f$ is log flat at $x$.
\item $C(M_{X,x})\to C(M_{Y,f(x)})$ is flat.
\end{enumerate}
\end{corollary}
\begin{proof}
By \Cref{flatstrata} (i), it suffices to prove (ii) implies (i). This is fppf local so, by \Cref{neatcharts} (v), we may assume the existence of a chart as in \Cref{goodcharts} and (i) follows.
\end{proof}

\begin{corollary}
Let $f:(X,M_X)\to (Y,M_Y)$ and $g:(Y,M_Y)\to (Z,M_Z)$ be morphisms of fine log schemes and $x\to X$ a geometric point. Assume $f$ and $g\circ f$ are log flat $x$. If $g$ is $s$-injective at $f(x)$ and $(Y,M_Y)$ is t-flat at $f(x)$, then $g$ is log flat at $f(x)$.
\end{corollary}
\begin{proof}
By \Cref{Kummerflatstrata} it suffice to show that $C(M_{Y,f(x)})\to C(M_{Z,(g\circ f)(x)})$ is flat. This follows from \Cref{flatstrata} (i).
\end{proof}

The previous result can be compared to \cite[0.2]{logflatdescent}.

\begin{corollary}\label{pointstrataflat}
Let $f:(X,M_X)\to (Y,M_Y)$ be a morphism of fine log schemes and $x\to X$ a geometric point such that $C(M_{Y,f(x)})$ is the spectrum of a field. If $f$ is s-injective at $x$ (e.g., Kummer) and $(X,M_X)$ is t-flat at $x$, then $f$ is log flat at $x$.
\end{corollary}
\begin{proof}
This follows from \Cref{Kummerflatstrata}.
\end{proof}

For a group scheme $G/S$, viewed as an fs log scheme with trivial log structure, we call a $G$-torsor in the log flat topology a \emph{log $G$-torsor}. A standard example, which we recall since we will use repeatedly, is the following.

\begin{example}\label{standardlogtorsor}
Let $h:Q\to P$ be Kummer homomorphism of fs monoids and $G=\D(\cok(h^{\gp}))$. There is a natural action of $G$ on $\D(P)$ and the associated morphism $\D(h):\D(P)\to\D(Q)$ is a log $G$-torsor. To see this, note that it is trivially log flat and, by \cite[3.4 (c)]{illusieoverview}, the action of $G$ on $\D(P)$ induces an isomorphism of fs log schemes
\[ G\times\D(P)\cong \D(P)\times_{\D(Q)}\D(P) \]
Moreover, since $h$ is a Kummer homomorphism of fs monoids, it is locally exact (\cite[I, 4.3.5 (3)]{logbook}), hence $\D(h)$ is an exact morphism (\cite[III, 2.2.2 (2)]{logbook}), which implies that $\D(h)$ is Kummer (cf. \cite[III, 2.2.1 (2)]{logbook}). Finally, $\D(h)$ is a covering in the log flat topology since its underlying morphism is of finite presentation and surjective (by `going up', being the spectrum of a finite injective homomorphism of rings).
\end{example}

\begin{corollary}\label{finiteflatdegree}
Suppose $G/Y$ is a finite flat group scheme and $f:(X,M_X)\to (Y,M_Y)$ a morphism of fs log schemes that is a log $G$-torsor. If $X\to Y$ is finite, then, for any geometric point $x\to X$, $C(M_{X,x})\to C(M_{Y,f(x)})$ is finite flat of degree dividing the order of $G$. If $G$ has prime order, then $C(M_{X,x})\to C(M_{Y,f(x)})$ is either an isomorphism or a $G$-torsor.
\end{corollary}

We will need a lemma (cf. \cite[2.7 (2)]{logflat}).

\begin{lemma}\label{katolemma}
Suppose $f:(X,M_X)\to (Y,M_Y)$ is a Kummer morphism of fs log schemes and $Q\to M_Y$ a chart that is neat at a geometric point $y\to Y$.
\begin{enumerate}[label=\normalfont(\roman*)]
\item If $X$ is quasi-compact, then there is an integer $n\geq 1$, such that, denoting $\mathbf{n}:\D(Q)\to\D(Q)$ the morphism induced by elevation to the power $n$ on $Q$, the fs base change $f\times_{\D(Q),\mathbf{n}}\D(Q)$ is strict at the points above $y$.
\item If $f$ is log \'etale, then, in \textup{(i)}, we can take $n$ to be invertible in a neighbourhood of $y$.
\item If $x\to X$ is a geometric point above $y$, then the fs base change $f\times_{\D(Q)}\D(\overline{M}_{X,x})$ is strict at geometric points above $x$.
\end{enumerate}
\end{lemma}
\begin{proof}
We first show (iii). Let $N$ be the log structure of $(X,M_X)\times_{\D(Q)}\D(\overline{M}_{X,x})$ and $x'$ a geometric point of the underlying scheme with $\pr_1(x')=x$. Let $P=\overline{M}_{X,x}$. If $z:=\pr_2(x')\to\D(P)$, then $\overline{M}_{\D(P),z}\cong \overline{P_{\mathfrak{p}}}$ for some prime ideal $\mathfrak{p}\subset P$ (\cite[II.1.4 (1)]{tsuji}). By \cite[2.1]{nakayamalogetale}, we have
\[ \overline{N}_{x'}\cong (P\oplus_Q\overline{P_{\mathfrak{p}}})/\text{torsion} \]
where the coproduct is taken in the category of fs monoids. Since $f$ is Kummer and $P$ fine, there is an integer $n_x\geq 1$ such that $p^{n_x}\in\text{im}(Q)$ for all $p\in P$, so by \cite[3.3]{illusieoverview} we have
\[ P\oplus_QP\cong (P^{\gp}/Q^{\gp})\oplus P \]
Thus,
\[ \overline{N}_{x'}\cong (P\oplus_QP\oplus_P\overline{P_{\mathfrak{p}}})/\text{torsion}\cong \overline{P_{\mathfrak{p}}} \]
Since $Y\times_{\D(Q)}\D(P)\to\D(P)$ is strict, this proves (iii). Note that, if $n_x\mid n$, then elevation to the power $n$ on $Q$ factors over $P$, hence $f\times_{\D(Q),\mathbf{n}}\D(Q)$ is also strict at the points above $x$. It follows that $f\times_{\D(Q),\mathbf{n_x}}\D(Q)$ is strict over an \'etale neighbourhood $U_x$ of $x$. If $X$ is quasi-compact, then finitely many $U_x$ cover $X$, so (i) follows by taking $n$ to be the lcm of the integers $n_x$ for these $U_x$. Finally, if $f$ is log \'etale, then the cokernel of $Q^{\gp}\to P^{\gp}$ is annihilated by an integer $m$ invertible at $x$ (cf. \cite[IV, 2.3.5]{logbook}). Since $Q\to P$ is exact (\cite[I, 4.3.5 (3)]{logbook}), we may take $n_x=m$, hence (ii).
\end{proof}

\begin{proof}[Proof of \Cref{finiteflatdegree}]
By \Cref{logstratacharts} (ii), the claim is local for the fppf topology at $f(x)\to Y$, so we may assume $Y$ is strictly local and there is a neat chart $Q\to M_Y$ (\Cref{neatcharts}). Let $(Y',M_{Y'})=(Y,M_Y)\times_{\D(Q),\mathbf{n}}\D(Q)$ as in \Cref{katolemma} (i), $(X',M_{X'}):=(X,M_X)\times_{(Y,M_Y)}(Y',M_{Y'})$, and $f':(X',M_{X'})\to (Y',M_{Y'})$ the morphism deduced from $f$. Then, $f'$ is a strict log $G$-torsor, so choosing a geometric point $x'\to X'$ lying above $x$, it follows that $C(M_{X',x'})$ is a connected component of $C(M_{Y',f'(x')})\times_{Y'}X'$ (\Cref{logstratacharts} (ii)), hence $C(M_{X',x'})\to C(M_{Y',f'(x')})$ is finite flat. Now, the degree of a connected component of $G$-torsor over a strictly local base is equal to the order of $G^0$, in particular divides the order of $G$, so the same holds for the degree of $C(M_{X',x'})\to C(M_{Y',f'(x')})$. Since $Q\to M_Y$ is neat, the map $C(M_{Y',f'(x')})\to C(M_{Y,f(x)})$ is an isomorphism (\Cref{logstratacharts} (iii)). Thus, the degree of $C(M_{X,x})\to C(M_{Y,f(x)})$ divides that of $C(M_{X',x'})\to C(M_{Y',f'(x')})$, which proves the first assertion. For the second, we may assume $C(M_{X',x'})\to C(M_{Y',f'(x')})$ is not an isomorphism, hence has prime degree. In this case, the closed immersion $C(M_{X',x'})\to C(M_{Y',f'(x')})\times_{Y'}X'$ is an isomorphism for degree reasons, hence $C(M_{X',x'})\to C(M_{Y',f'(x')})$ is a $G$-torsor. If $C(M_{X,x})\to C(M_{Y,f(x)})\cong C(M_{Y',f'(x')})$ is not an isomorphism, we must have $C(M_{X',x'})\cong C(M_{X,x})$ and the assertion follows.
\end{proof}

\subsection{Logarithmic dimension and logarithmic depth}

\begin{definition}
The \emph{logarithmic dimension} of the strict localisation of a fine log scheme $(X,M_X)$ at a geometric point $x\to X$, denoted $\dim(X_{(x)},M_X)$, is defined
\[ \dim(X_{(x)},M_X)=\rk\overline{M}_{X,x}^{\gp}+\dim C(M_{X,x}) \]
\end{definition}

\begin{remark} This is denoted $d(x,M_X)$ in \cite[12.5.5]{gabberramero}. 
\end{remark}

\begin{proposition}\label{tflatprop}
Let $f:(X,M_X)\to (Y,M_Y)$ be a morphism of fine log scheme, $x\to X$ a geometric point, $y=f(x)$, and $K:=\ker(\overline{M}_{Y,y}^{\gp}\to\overline{M}_{X,x}^{\gp})$.
\begin{enumerate}[label=\normalfont(\roman*)]
\item If $f$ is a closed immersion of log schemes \textup{(\cite[3.1]{log})}, then $\dim(X_{(x)},M_X)\leq\dim(Y_{(y)},M_Y)$.
\item If $f$ is log \'etale, then $\dim(X_{(x)},M_X)=\dim(Y_{(y)},M_Y)-\td (k(x)/k(y))$. Moreover, $\td (k(x)/k(y))=0$ if $K$ is torsion.
\item If $f$ is log flat and $K$ is torsion, then $\dim(X_{(x)},M_X)\geq\dim(Y_{(y)},M_Y)$. If, in addition, $f$ is small\footnote{in the sense of \cite[III, 2.3.2]{logbook}} at $x$ and $X\to Y$ quasi-finite, then $\dim(X_{(x)},M_X)=\dim(Y_{(y)},M_Y)$.
\end{enumerate}
\end{proposition}
\begin{proof}
(i) is immediate from the definitions. For the remaining assertions, note that if $f$ is small at $x$ (e.g., log \'etale, cf. \Cref{neatcharts} (vii)), then $\cok(\overline{M}_{Y,y}^{\gp}\to\overline{M}_{X,x}^{\gp})$ is torsion, so $\rk\overline{M}_{Y,y}^{\gp}=\rk\overline{M}_{X,x}^{\gp}+\rk\ker(\overline{M}_{Y,y}^{\gp}\to\overline{M}_{X,x}^{\gp})$. Now (ii) follows from \Cref{flatstrata} (ii). In (iii), we have $\rk\overline{M}_{X,x}^{\gp}\geq\rk\overline{M}_{Y,y}^{\gp}$ with equality for $f$ small, so (iii) reduces to showing $\dim C(M_{X,x})\geq\dim C(M_{Y,y})$ with equality for $X\to Y$ quasi-finite. This follows from \Cref{flatstrata} (i) and \cite[IV\textsubscript{2}, 6.1.2]{ega}.
\end{proof}

\begin{corollary}\label{ldimlogblowup}
If $f:(X,M_X)\to (Y,M_Y)$ is a log blowup of fine log schemes, then
\[ \dim(X_{(x)},M_X)=\dim(Y_{(f(x))},M_Y)-\td (k(x)/k(f(x))) \]
for any geometric point $x\to X$.
\end{corollary}
\begin{proof}
Since log blowups of fine log schemes are log \'etale (cf. \cite[12.4.43]{gabberramero}), the assertion follows from \Cref{tflatprop} (ii).
\end{proof}

\begin{corollary}\label{ldimsaturation}
If $f:(X,M_X)\to (Y,M_Y)$ is the saturation of a fine log scheme, then $\dim(X_{(x)},M_X)=\dim(Y_{(f(x))},M_Y)$ for any geometric point $x\to X$.
\end{corollary}
\begin{proof}
Indeed, $f$ is finite log \'etale so this follows from \Cref{tflatprop} (ii).
\end{proof}

\begin{theorem}\label{ldimineq}
Let $(X,M_X)$ be a fine log scheme and $x\to X$ a geometric point. We have $\dim(X_{(x)},M_X)\geq\dim X_{(x)}$, with equality if $(X,M_X)$ is t-flat at $x$.
\end{theorem}
\begin{proof}
By \Cref{tflatprop} (iii), we may replace $X$ by a quasi-finite flat covering. Thus, by \Cref{neatcharts} (ii), we may assume $X$ is strictly local and there is a neat chart $P\to M_{X}$ at $x$. Then, $P$ is sharp and $P\to\mathcal{O}_{X,x}$ is local, so we can apply \Cref{diminequality}.
\end{proof}

\begin{corollary}\label{trivdense}
Let $(X,M_X)$ be a t-flat fine log scheme. If $\overline{M}_{X}^{\gp}$ is torsion free, then $X_{\triv}\subset X$ is dense.
\end{corollary}
\begin{proof}
It suffices to show that $X_{\triv}$ contains the generic point of every irreducible component. If $x\to X$ lies above a generic point, then by \Cref{ldimineq} we have $\dim(X_{(x)},M_X)=\dim X_{(x)}=0$ , hence $\rk\overline{M}_{X,x}^{\gp}=0$, whence $\overline{M}_{X,x}^{\gp}=0$ since $\overline{M}_{X,x}^{\gp}$ is torsion free. Thus, $\overline{M}_{X,x}=\{1\}$ and $x\to X_{\triv}$.
\end{proof}

\begin{corollary}\label{ldimineqgen}
Let $(X,M_X)$ be a fine log scheme and $x\to X$ a geometric point. For any geometric point $y\to X_{(x)}$, we have $\dim (X_{(y)},M_X)\leq\dim (X_{(x)},M_X)$.
\end{corollary}
\begin{proof}
We may assume $X=X_{(x)}$ and there is an exact fine chart $P\to M_{X,x}$ (so that $\overline{P}\cong\overline{M}_{X,x}$). Then, the log structure $M_X$ is Zariski (\cite[III, 1.4.1 (2)]{logbook}) and it suffices to show the assertion for a point $y\in X$. By \cite[II.1.4]{tsuji}, $\overline{P_{\mathfrak{p}}}\cong\overline{M}_{X,y}$, where $\mathfrak{p}\subset P$ is the (prime) inverse image of the maximal ideal of $\mathcal{O}_{X,y}$, and $C(M_{X,y})\subset X_{(y)}$ is the closed subscheme of ideal generated by the image of $\mathfrak{p}$ (\Cref{logstratacharts} (i)).

Let $Y\subset X$ be the closed subscheme of ideal $\mathfrak{p}\mathcal{O}_{X,x}$; then, $C(M_{X,y})=Y_{(y)}$. Let $M_Y$ denote the fine log structure on $Y$ associated to $P\setminus\mathfrak{p}\to\mathcal{O}_Y$. Note that $P\setminus\mathfrak{p}\to\mathcal{O}_{Y,x}$ is local (since $P\to\mathcal{O}_{X,x}$ is), hence so is $P\setminus\mathfrak{p}\to M_{Y,x}$; this implies $\overline{P\setminus\mathfrak{p}}\cong\overline{M}_{Y,x}$. By \cite[I, 1.4.7]{logbook}, we have $\rk\overline{P}^{\gp}=\rk\overline{P_{\mathfrak{p}}}^{\gp}+\rk\overline{P\setminus\mathfrak{p}}^{\gp}$, hence
\begin{equation}\label{primeidealequality}
\rk\overline{M}_{X,x}^{\gp}=\rk\overline{M}_{X,y}^{\gp}+\rk\overline{M}_{Y,x}^{\gp}
\end{equation}
Since $C(M_{Y,x})=C(M_{X,x})$, we have $\dim (Y_{(x)},M_Y)=\dim C(M_{X,x})+\rk\overline{M}_{Y,x}^{\gp}$. By \Cref{ldimineq},
$\dim (Y_{(x)},M_Y)\geq\dim Y$, so, since $\dim Y\geq \dim Y_{(y)}=\dim C(M_{X,y})$, we obtain
\[ \dim C(M_{X,y})\leq \dim C(M_{X,x})+\rk\overline{M}_{Y,x}^{\gp} \]
By \cref{primeidealequality} we get
\[ \dim C(M_{X,y})+\rk\overline{M}_{X,y}^{\gp}\leq \dim C(M_{X,x})+\rk\overline{M}_{X,x}^{\gp}=\dim (X_{(x)},M_X) \]
as required.
\end{proof}

\begin{corollary}\label{logetaledim}
Let $f:(X,M_X)\to (Y,M_Y)$ be a log \'etale morphism and $x\to X$ a geometric point. If $(Y,M_Y)$ is t-flat at $f(x)$, then
\[ \dim X_{(x)}=\dim Y_{(f(x))}-\td (k(x)/k(f(x))) \]
Moreover, $\td (k(x)/k(f(x)))=0$ if $\ker(\overline{M}_{Y,f(x)}^{\gp}\to\overline{M}_{X,x}^{\gp})$ is torsion.
\end{corollary}
\begin{proof}
Indeed, $(X,M_X)$ is t-flat at $x$ (\Cref{klftlf}), hence the result follows from  \Cref{ldimineq} and \Cref{tflatprop} (ii).
\end{proof}

\begin{corollary}\label{dimlogstrata}
Let $f:(X,M_X)\to (Y,M_Y)$ be a morphism of t-flat fine log schemes and $x\to X$ a geometric point such that $\dim X_{(x)}=\dim Y_{(f(x))}$. If $\rk\overline{M}_{X,x}^{\gp}=\rk\overline{M}_{Y,f(x)}^{\gp}$, then $\dim C(M_{X,x})=\dim C(M_{Y,f(x)})$.
\end{corollary}
\begin{proof}
This follows from \Cref{ldimineq}.
\end{proof}

\begin{corollary}\label{logregcharacterisation}
An fs log scheme $(X,M_X)$ is log regular at $x\to X$ if and only if the following conditions hold
\begin{enumerate}[label=\normalfont(\arabic*)]
\item $(X,M_X)$ is t-flat at $x$
\item $C(M_{X,x})$ is regular.
\end{enumerate}
\end{corollary}
\begin{proof}
By definition (\cite{logreg}, \cite[12.5.7]{gabberramero}), $(X,M_X)$ is log regular at $x\to X$ if and only if $C(M_{X,x})$ is regular and $\dim(X_{(x)},M_X)=\dim X_{(x)}$. So, by \Cref{ldimineq}, conditions (1) and (2) are sufficient. Conversely, since log regularity implies t-flatness (\cite[6.1]{logreg}), they are also necessary.
\end{proof}

This motivates the following definition.

\begin{definition}\label{tregdef}
We say that a fine log scheme $(X,M_X)$ is \emph{t-regular at a geometric point $x\to X$} if it is t-flat at $x$ and $C(M_{X,x})$ is regular.
\end{definition}

\begin{proposition}\label{logetaleregular}
Let $f:(X,M_X)\to (Y,M_Y)$ be a log flat morphism of a fine log schemes and $x\to X$ a geometric point.
\begin{enumerate}[label=\normalfont(\roman*)]
\item If $(X,M_X)$ is t-regular at $x$ and $(Y,M_Y)$ t-flat at $f(x)$, then $(Y,M_Y)$ is t-regular at $f(x)$.
\item Assume $f$ is log \'etale. If $(Y,M_Y)$ is t-regular at $f(x)$ and $\overline{M}_{Y,f(x)}^{\gp}$ is torsion free, then $(X,M_X)$ is t-regular at $x$.
\end{enumerate}
Assume that, locally for the fppf topology, there is a neat Kummer chart $Q\to P$ of fine monoids for $f$ at $x$.
\begin{enumerate}[label=\normalfont(\roman*),resume]
\item If $(X,M_X)$ is t-regular at $x$, then $(Y,M_Y)$ is t-regular at $f(x)$.
\end{enumerate}
\end{proposition}
\begin{proof}
Since $C(M_{X,x})\to C(M_{Y,f(x)})$ is flat (\Cref{flatstrata}), $C(M_{Y,f(x)})$ is regular if $C(M_{X,x})$ is (\cite[0\textsubscript{IV}, 17.3.3]{ega}), hence (i). In (ii), $(X,M_X)$ is t-flat at $x$ (\Cref{klftlf}), so the assertion follows from \Cref{flatstrata} (iii).

For (iii), note that, by \Cref{klftlf}, under the stated assumption $(X,M_X)$ is t-flat at $x$ if and only if $(Y,M_Y)$ is t-flat at $f(x)$. So (iii) follows from (i).
\end{proof}

This implies the following result due to Gabber (cf. \cite[3.5]{gabberVIII}).

\begin{corollary}\label{saturationregular}
Let $f:(X,M_X)\to (Y,M_Y)$ be the saturation of a fine log scheme. If $(X,M_X)$ is log regular at $x\to X$, then $(Y,M_Y)$ is t-regular at $f(x)$. The converse holds if, in addition, $\overline{M}_{Y,f(x)}^{\gp}$ is torsion free.
\end{corollary}
\begin{proof}
Indeed, locally for the \'etale topology at $f(x)$ there is a fine chart $Q\to M_Y$ that is exact at $f(x)$ (\Cref{neatcharts} (i)) and, by \Cref{saturationtorsion} (ii), $Q\to Q^{\sat}$ is a neat Kummer chart of fine monoids for $f$ at $x$, so the assertions follow from \Cref{logetaleregular} (ii)--(iii).
\end{proof}

\begin{corollary}
Let $(X,M_X)$ be a fine log scheme and $x\to X$ a geometric point. If $\overline{M}_{X,x}^{\gp}$ is torsion free, then the following are equivalent.
\begin{enumerate}[label=\normalfont(\roman*)]
\item $(X,M_X)$ is t-regular at $x$.
\item $\dim X_{(x)}=\dim (X_{(x)},M_X)$ and $C(M_{X,x})$ is regular.
\end{enumerate}
\end{corollary}
\begin{proof}
By \Cref{ldimineq}, it is enough to show (ii) $\Rightarrow$ (i);  equivalently, by \Cref{saturationregular}, the saturation of $(X,M_X)$ is log regular at the geometric points above $x$. For this, it suffices to show that (ii) also holds for the saturation of $(X,M_X)$ at the geometric points above $x$. The first condition follows from \Cref{saturationtflat} and \Cref{ldimineq}, and since $\overline{M}_{X,x}^{\gp}$ is torsion free and saturation is log \'etale, the second follows from \Cref{saturationtorsion} (iii) and \Cref{flatstrata} (iii).
\end{proof}

\begin{remark}
Using \Cref{saturationregular}, one can extend properties of log regularity to t-regular log schemes $(X,M_X)$ with $\overline{M}_X^{\gp}$ is torsion free, cf. \cite{thompson}.
\end{remark}

We now come to the logarithmic depth.

\begin{definition}
Let $(X,M_X)$ be a fine log scheme. The \emph{logarithmic depth} of $(X,M_X)$ at a geometric point $x\to X$, denoted $\depth(X_{(x)},M_X)$, is defined
\[ \depth(X_{(x)},M_X)=\rk\overline{M}^{\gp}_{X,x}+\depth (\mathcal{O}_{C(M_{X,x}),x}) \]
\end{definition}

Clearly, $\depth(X_{(x)},M_X)\leq\dim (X_{(x)},M_X)$.

\begin{proposition}\label{ldepthlogetale}
If $f:(X,M_X)\to (Y,M_Y)$ is a log \'etale morphism of fine log schemes, then
\[ \depth(X_{(x)},M_X)=\depth(Y_{(f(x))},M_Y)-\td (k(x)/k(f(x))) \]
for any geometric point $x\to X$. Moreover, $\td (k(x)/k(f(x)))=0$ if $\ker(\overline{M}_{Y,f(x)}^{\gp}\to\overline{M}_{X,x}^{\gp})$ is torsion.
\end{proposition}
\begin{proof}
Since $\cok(\overline{M}_{Y,f(x)}^{\gp}\to\overline{M}_{X,x}^{\gp})$ is torsion (cf. \Cref{neatcharts} (vii)), we have $\rk\overline{M}_{Y,f(x)}^{\gp}=\rk\overline{M}_{X,x}^{\gp}+\rk\ker(\overline{M}_{Y,f(x)}^{\gp}\to\overline{M}_{X,x}^{\gp})$, and the assertion follows from \Cref{flatstrata} (ii).
\end{proof}

\begin{corollary}\label{ldepthlogblowup}
If $f:(X,M_X)\to (Y,M_Y)$ is a log blowup of fine log schemes, then $\depth(X_{(x)},M_X)=\depth(Y_{(f(x))},M_Y)-\td (k(x)/k(f(x)))$ for any geometric point $x\to X$.
\end{corollary}
\begin{proof}
Since log blowups are log \'etale, this follows from \Cref{ldepthlogetale}.
\end{proof}

\begin{corollary}\label{ldepthsaturation}
If $f:(X,M_X)\to (Y,M_Y)$ is the saturation of a fine log scheme, then $\depth(X_{(x)},M_X)=\depth(Y_{(f(x))},M_Y)$ for any geometric point $x\to X$.
\end{corollary}
\begin{proof}
By \Cref{saturationtorsion} (iii), this follows from \Cref{ldepthlogetale}.
\end{proof}

The logarithmic depth behaves well in the log flat topology.

\begin{proposition}\label{ldepthrises}
If $f:(X,M_X)\to (Y,M_Y)$ is a log flat morphism of fine log schemes, then, at any geometric point $x\to X$ at which $\ker(\overline{M}_{Y,f(x)}^{\gp}\to\overline{M}_{X,x}^{\gp})$ is torsion, we have
\[ \depth(X_{(x)},M_X)\geq\depth(Y_{(f(x))},M_Y) \]
with equality if $f$ is small and $X\to Y$ quasi-finite.
\end{proposition}
\begin{proof}
By assumption, we have $\rk\overline{M}_{Y,f(x)}^{\gp}\leq\rk\overline{M}_{X,x}^{\gp}$, with equality if $f$ is small. So the assertion follows from \Cref{flatstrata} and \cite[IV\textsubscript{2}, 6.3.1]{ega}.
\end{proof}

Using the logarithmic dimension and depth, we can define a logarithmic analogue of Cohen--Macaulayness.

\begin{definition}
We say that a fine log scheme $(X,M_X)$ is \emph{log Cohen--Macaulay} at a geometric point $x\to X$ if
\begin{equation}\label{CMequality0}
\depth(X_{(y)},M_X)=\dim(X_{(y)},M_X)
\end{equation}
for every geometric point $y\to X_{(x)}$.
\end{definition}

\begin{lemma}\label{CMstrata}
$(X,M_X)$ is log Cohen--Macaulay at a geometric point $x\to X$ if and only if $C(M_{X,y})$ is Cohen--Macaulay for every geometric point $y\to X_{(x)}$.
\end{lemma}
\begin{proof}
Clear.
\end{proof}

\begin{corollary}
Let $f:(X,M_X)\to (Y,M_Y)$ be a log flat morphism of fine log schemes and $x\to X$ a geometric point such that $X\to Y$ is quasi-finite at $x$. If $(Y,M_Y)$ is log Cohen--Macaulay at $f(x)$, then $(X,M_X)$ is log Cohen--Macaulay at $x$.
\end{corollary}
\begin{proof}
Since $X\to Y$ is quasi-finite, we may assume $Y$ (resp. $X$) is strictly local of closed point $f(x)$ (resp. $x$) and $X\to Y$ finite (\cite[IV\textsubscript{4}, 18.5.11 (c')]{ega}). Then, for $y\to X$, $C(M_{X,y})\to C(M_{Y,f(y)})$ is finite and flat (\Cref{flatstrata}), hence $C(M_{X,y})$ is Cohen--Macaulay if and only if $C(M_{Y,f(y)})$ is. Now apply \Cref{CMstrata}.
\end{proof}

Over a t-regular log scheme, we have `miracle log flatness'.

\begin{proposition}\label{miracleflatness}
Let $f:(X,M_X)\to (Y,M_Y)$ be a Kummer morphism of fine log schemes and $x\to X$ a geometric point such that $X\to Y$ is quasi-finite at $x$ and $\dim X_{(x)}=\dim Y_{f(x)}$. If $(X,M_X)$ is log Cohen--Macaulay and t-flat at $x$, and $(Y,M_Y)$ is t-regular at $f(x)$, then $f$ is log flat at $x$.
\end{proposition}
\begin{proof}
By \Cref{Kummerflatstrata}, it suffices to show that $C(M_{X,x})\to C(M_{Y,f(x)})$ is flat. Note that it is quasi-finite (since $X\to Y$). Since $C(M_{Y,f(x)})$ is regular and $C(M_{X,x})$ is Cohen--Macaulay (\Cref{CMstrata}), by \cite[IV\textsubscript{1}, 17.3.5]{ega} it suffices to show that $\dim C(M_{X,x})=\dim C(M_{Y,f(x)})$. By \Cref{dimlogstrata}, this follows from the assumptions.
\end{proof}

The next result extends \Cref{depthequality} to log schemes.

\begin{theorem}\label{depthequality2}
Let $(X,M_X)$ be a fine log scheme. If $(X,M_X)$ is t-flat at $x\to X$, then $\depth(X_{(x)},M_X)\geq \depth(\mathcal{O}_{X,x})$, with equality if $M_X$ is saturated.
\end{theorem}
\begin{proof}
Up to passing to a quasi-finite flat covering (which affects neither the depth nor the logarithmic depth, cf. \Cref{ldepthrises}), by \Cref{neatcharts} we may assume there is a neat chart for $M_X$ and we can apply \Cref{depthequality}.
\end{proof}

\begin{corollary}\label{CM}
Let $(X,M_X)$ be an fs log scheme and $x\to X$ be a geometric point. If $(X,M_X)$ is t-flat at $x$, then the following are equivalent.
\begin{enumerate}[label=\normalfont(\roman*)]
\item $(X,M_X)$ is log Cohen--Macaulay at $x$.
\item $C(M_{X,x})$ is Cohen--Macaulay.
\item $X_{(x)}$ is Cohen--Macaulay.
\end{enumerate}
\end{corollary}
\begin{proof}
This follows from \Cref{gentflat}, \Cref{ldimineq}, \Cref{CMstrata} and \Cref{depthequality2}.
\end{proof}

\begin{corollary}
Let $f:(X,M_X)\to (Y,M_Y)$ be a log \'etale morphism of fs log schemes and $x\to X$ a geometric point. If $(Y,M_Y)$ is t-flat at $f(x)$, then
\[ \depth(\mathcal{O}_{X,x})=\depth(\mathcal{O}_{Y,f(x)})-\td (k(x)/k(f(x))) \]
Moreover, $\td (k(x)/k(f(x)))=0$ if $\ker(\overline{M}_{Y,f(x)}^{\gp}\to\overline{M}_{X,x}^{\gp})$ is torsion.
\end{corollary}
\begin{proof}
This follows from \Cref{klftlf}, \Cref{ldepthlogetale} and \Cref{depthequality2}.
\end{proof}

\begin{corollary}\label{klfdepth}
Let $f:(X,M_X)\to (Y,M_Y)$ be a log flat morphism of fine log schemes. If $M_X$ is saturated and $(Y,M_Y)$ is t-flat, then
\[ \depth(\mathcal{O}_{X,x})\geq\depth(\mathcal{O}_{Y,f(x)}) \]
for any geometric point $x\to X$ at which $\ker(\overline{M}_{Y,f(x)}^{\gp}\to\overline{M}_{X,x}^{\gp})$ is torsion, with equality if $f$ is small, $X\to Y$ is quasi-finite and $M_Y$ is saturated.
\end{corollary}
\begin{proof}
This follows from \Cref{klftlf}, \Cref{ldepthrises} and \Cref{depthequality2}.
\end{proof}

In particular, for any surjective Kummer log flat morphism of t-flat fs log schemes, the depth of the source is always at least that of the target. For quasi-finite morphisms this implies the following.

\begin{corollary}
Let $f:(X,M_X)\to (Y,M_Y)$ be a Kummer log flat morphism of fs log schemes with $X\to Y$ quasi-finite. If $(Y,M_Y)$ is t-flat and $Y$ is Cohen--Macaulay, then $X$ is Cohen--Macaulay. In particular, if $Y$ is regular, then $X\to Y$ is flat.
\end{corollary}
\begin{proof}
The first assertion follows from \Cref{klfdepth}. For the second,  we may assume $X$ and $Y$ strictly local and $X\to Y$ finite, in which case it follows from the first assertion and the Auslander--Buchsbaum formula.
\end{proof}

More generally, we can define a logarithmic analogue of Serre's $S_n$ property.

\begin{definition}
We say that a fine log scheme $(X,M_X)$ is \emph{log $S_n$} at a geometric point $x\to X$, if
\[  \depth(X_{(y)},M_X)\geq\min\{n,\dim(X_{(y)},M_X)\} \]
for every geometric point $y\to X_{(x)}$.
\end{definition}

Here is a variant of \Cref{CM}.

\begin{corollary}\label{logSn}
Assume $(X,M_X)$ is fs and t-flat at a geometric point $x\to X$. $(X,M_X)$ is log $S_n$ at $x$ if and only if $X$ is $S_n$ at $x$.
\end{corollary}
\begin{proof}
By \Cref{gentflat}, this follows from \Cref{depthequality2} and \Cref{ldimineq}.
\end{proof}

\begin{corollary}
Assume $(X,M_X)$ is fs and t-flat at a geometric point $x\to X$. If $X_{\triv}$ is $S_1$, then $X$ is $S_1$ at $x$.
\end{corollary}
\begin{proof}
Indeed, by \Cref{gentflat} and \Cref{depthequality2} we have $\depth(\mathcal{O}_{X,y})\geq 1$ for any generisation $y$ of $x$ with $y\not\in X_{\triv}$.
\end{proof}

\begin{proposition}\label{logSnrises}
Let $f:(X,M_X)\to (Y,M_Y)$ be a log flat morphism of fine log schemes and $x\to X$ a geometric point. Assume $f$ is small, $X\to Y$ is quasi-finite, and $\ker(\overline{M}_{Y,f(y)}^{\gp}\to\overline{M}_{X,y}^{\gp})$ is torsion for any geometric point $y\to X_{(x)}$. If $(Y,M_Y)$ is log $S_n$ at $f(x)$, then $(X,M_X)$ is log $S_n$ at $x$.
\end{proposition}
\begin{proof}
By \Cref{ldepthrises}, it suffices to show $\dim(X_{(y)},M_X)=\dim(Y_{(f(y))},M_Y)$ for any geometric point $y\to X_{(x)}$. This follows from \Cref{tflatprop} (iii).
\end{proof}

\subsection{Log cotangent complex and t-flatness}\label{logcotangentsection}
In \cite{logcotangent}, Olsson gives two different constructions of log cotangent complexes for fine log schemes, one of which is due to Gabber, and provides conditions for the two complexes to be the same. We begin by remarking that t-flatness ensures this.

\begin{proposition}\label{GabberOlsson}
Let $f:(X,M_X)\to (Y,M_Y)$ be a morphism of fine log schemes. If $(Y,M_Y)$ is t-flat, then the canonical map
\[ \tau_{\geq n}L_{(X,M_X)/(Y,M_Y)}^G\to \tau_{\geq n}L_{(X,M_X)/(Y,M_Y)}^O \]
from Gabber's log cotangent complex to Olsson's log cotangent complex is a quasi-isomorphism for all $n\in\mathbb{Z}$.
\end{proposition}
\begin{proof}
Let $x\to X$ be a geometric point and $y=f(x)$. The claim is \'etale local; let $Q\to P$ be a fine chart for $f$ at $x$. By \Cref{neatcharts} (iii), we may assume $Q\to P$ to be injective. By \cite[proof of 8.32]{logcotangent} it suffices to show that the base change morphism of (classical) cotangent complexes
\[ \mathcal{O}_{X,x}\otimes_{\mathbb{Z}[P]}L_{\mathbb{Z}[P]/\mathbb{Z}[Q]}\to \mathcal{O}_{X,x}\otimes_{\mathcal{O}_{Y,y}\otimes_{\mathbb{Z}[Q]}\mathbb{Z}[P]}L_{\mathcal{O}_{Y,y}\otimes_{\mathbb{Z}[Q]}\mathbb{Z}[P]/\mathcal{O}_{Y,y}} \]
is a quasi-isomorphism. By \cite[II, 2.2.1]{cotangent1}, it suffices that $\Tor_i^{\mathbb{Z}[Q]}(\mathbb{Z}[P],\mathcal{O}_{Y,y})=0$ for $i>0$, which follows from \Cref{primeideals} (ii).
\end{proof}

\begin{corollary}
Let $(X,M_X)\to (Y,M_Y)$ be a log smooth morphism of fine log schemes. If $(Y,M_Y)$ is t-flat, then $L_{(X,M_X)/(Y,M_Y)}^G\cong\Omega^1_{(X,M_X)/(Y,M_Y)}[0]$.
\end{corollary}
\begin{proof}
This follows from \Cref{GabberOlsson} and \cite[1.1 (iii)]{logcotangent}.
\end{proof}

In the sequel, for a morphism $(X,M_X)\to (Y,M_Y)$ of fine log schemes we will denote by $L_{(X,M_X)/(Y,M_Y)}$ Gabber's log cotangent complex. The main advantage of Gabber's complex is that it satisfies the usual distinguished triangles for compositions of morphisms of log schemes (\cite[8.29]{logcotangent}). For trivial log structures we will omit the log structures from notation; in this case, by \cite[8.21]{logcotangent} there is a canonical quasi-isomorphism with the classical cotangent complex, so there is essentially no ambiguity.

\begin{lemma}\label{strictcotangent}
For $(X,M_X)\to (Y,M_Y)$ a strict morphism of fine log schemes, there is a functorial quasi-isomorphism $L_{(X,M_X)/(Y,M_Y)}\cong L_{X/Y}$.
\end{lemma}
\begin{proof}
The composition $L_{X/Y}\to L_{(X,M_X)/Y}\to L_{(X,M_X)/(Y,M_Y)}$ is a quasi-isomor\-phism: this is local on $X$, so we reduce to the case of log rings shown in \cite[\S 8]{logcotangent}.
\end{proof}

The main result of this section, \Cref{distinguished} below, is a characterisation of t-flatness by the vanishing of the homology of a certain object of the derived category defined using log cotangent complexes. In the sequel we work in derived categories, so in particular, for a morphism of schemes, the inverse image functor on complexes is meant in the derived sense.

Let $(X,M_X)$ be a fine log scheme, $x\to X$ a geometric point, $C=C(M_{X,x})$ the logarithmic stratum through $x$, and $i:C\to X$ the canonical morphism. If $\overline{M}_{X,x}^{\gp}$ is torsion free, we can define an object $L_{C/(X,M_X)}$ as follows. Let $P\to M_X$ be a neat chart at $x$ (it exists in an \'etale neighbourhood since $\overline{M}_{X,x}^{\gp}$ is torsion free), $j:\mathcal{S}\to\D(P)$ the strict closed immersion of log schemes defined by the ideal $(P\setminus\{1\})\subset\mathbb{Z}[P]$, and $g:(C,i^*M_X)\to \mathcal{S}$ the induced morphism. Note that $g$ is strict and the underlying scheme of $\mathcal{S}$ is $\Spec(\mathbb{Z})$.

We have canonical morphisms
\[ i^*L_{(X,M_X)/\mathbb{Z}}\to L_{(C,i^*M_X)/\mathbb{Z}}\to L_{(C,i^*M_X)/\mathcal{S}} \]
and a quasi-isomorphism $L_{(C,i^*M_X)/\mathcal{S}}\cong L_{C/\mathbb{Z}}$, since $(C,i^*M_X)\to\mathcal{S}$ is strict. So we obtain a morphism
\[ i^*L_{(X,M_X)/\mathbb{Z}}\to L_{C/\mathbb{Z}} \]
and we define $L_{C/(X,M_X)}$ as its cone, so that we have a distinguished triangle
\[ i^*L_{(X,M_X)/\mathbb{Z}}\to L_{C/\mathbb{Z}}\to L_{C/(X,M_X)} \]
More generally, for any morphism $s:B\to C$, we define $L_{B/(X,M_X)}$ by the distinguished triangle
\begin{equation}\label{Lxdef}
s^*i^*L_{(X,M_X)/\mathbb{Z}}\to L_{B/\mathbb{Z}}\to L_{B/(X,M_X)}
\end{equation}
and then we have a distinguished triangle
\[  s^*L_{C/(X,M_X)}\to L_{B/(X,M_X)}\to L_{B/C} \]
In particular, we will be interested in the case $B\subset C$ is the closed point.

\begin{remark}\label{logcotangentremark}
The notation $L_{C/(X,M_X)}$ suggests it is the cotangent complex of a morphism $C\to (X,M_X)$. Although there is no such morphism, one can formally adjoin $0$ to the trivial log structure of $C$ to obtain a log structure $N_C$ and a canonical morphism $(C,N_C)\to (X,M_X)$, and $L_{C/(X,M_X)}$ should be the log cotangent complex of this morphism. Since this requires the introduction of non-integral log schemes, we do not pursue this here, but see \cite[\S 6]{bhattderived} and \cite[\S 4]{SSV} for the case of log rings.
\end{remark}

\begin{theorem}\label{distinguished}
Assume $\overline{M}^{\gp}_{X,x}$ torsion free. The following are equivalent.
\begin{enumerate}[label=\normalfont(\roman*)]
\item $(X,M_X)$ is t-flat at $x$.
\item $H^iL_{C/(X,M_X)}=0$ for $i\neq -1$.
\item $H^{-2}L_{C/(X,M_X)}=0$
\end{enumerate}
Moreover, if this holds, then $L_{C/(X,M_X)}\cong \overline{M}_X^{\gp}\otimes_{\mathbb{Z}}\mathcal{O}_C[1]$.
\end{theorem}

\begin{remark}
By analogy with the classical cotangent complex characterisation of regular immersions (\cite[III, 3.2.4]{cotangent1}), \Cref{distinguished} should be interpreted as saying that $(X,M_X)$ is t-flat at $x$ if and only if $(C,N_C)\to (X_{(x)},M_X)$ is a `log regular' immersion, where $N_C$ is the log structure associated to $0\to\mathcal{O}_C$ (\Cref{logcotangentremark})---cf. \cite{cotangentlogreg}. Thus, $\overline{M}_X^{\gp}\otimes_{\mathbb{Z}}\mathcal{O}_C$ should be the `log conormal' sheaf of this morphism. In light of \Cref{depthequality2}, this suggests that, homologically, $(X,M_X)$ is locally at $x$ a vector bundle over $(C,N_C)$.
\end{remark}

For the proof of \Cref{distinguished} we begin with a general result.

\begin{proposition}\label{canonicalmap}
\begin{enumerate}[label=\normalfont(\roman*)]
\item There is a canonical morphism
\[ i^*L_{(X,M_X)/\mathbb{Z}}\to \overline{M}_{X}^{\gp}\otimes_{\mathbb{Z}}\mathcal{O}_C[0] \]
If $\overline{M}_{X,x}^{\gp}$ is torsion free, then this morphism has a section.
\item For any morphism $s:B\to C$, the morphism of \textup{(i)} induces a morphism $L_{B/(X,M_X)}\to \overline{M}_{X}^{\gp}\otimes_{\mathbb{Z}}\mathcal{O}_B[1]$.
\item If $\overline{M}_{X,x}^{\gp}$ is torsion free, then the morphism of \textup{(ii)} has a section inducing a quasi-isomorphism
\[ L_{B/(X,M_X)}\simeq K_B\oplus (\overline{M}_{X}^{\gp}\otimes_{\mathbb{Z}}\mathcal{O}_B)[1] \]
where $K_B$ is the cone of the canonical morphism
\[ s^*g^*L_{\mathcal{S}/\D(P)}\to L_{(B,s^*i^*M_X)/(X,M_X)}\cong L_{B/X} \]
with $P\to M_X$ a chart that is neat at $x$.
\end{enumerate}
\end{proposition}
\begin{proof}
Consider the homomorphism $\mathcal{O}_X\oplus M_X\to \overline{M}_{X}^{\gp}\otimes_{\mathbb{Z}}\mathcal{O}_C$ defined as the sum of the zero map on $\mathcal{O}_X$ and the canonical map $M_X\to\overline{M}_X^{\gp}$. One checks without difficulty that it is a log derivation, hence corresponds to a canonical homomorphism $\Omega^1_{(X,M_X)/\mathbb{Z}}\to \overline{M}_{X}^{\gp}\otimes_{\mathbb{Z}}\mathcal{O}_C$. Composing with the canonical map $L_{(X,M_X)/\mathbb{Z}}\to H^0L_{(X,M_X)/\mathbb{Z}}\cong\Omega^1_{(X,M_X)/\mathbb{Z}}$ (\cite[8.9]{logcotangent}), we obtain a canonical morphism
\[ i^*L_{(X,M_X)/\mathbb{Z}}\to \overline{M}_{X}^{\gp}\otimes_{\mathbb{Z}}\mathcal{O}_C[0] \]
Assume $\overline{M}_{X,x}^{\gp}$ is torsion free and let $P\to M_X$ be a chart that is neat at $x$. Since $L_{\D(P)/\mathbb{Z}}\cong P^{\gp}\otimes_{\mathbb{Z}}\mathcal{O}_{\D(P)}[0]$ (\cite[8.23]{logcotangent}), we have $g^*j^*L_{\D(P)/\mathbb{Z}}\cong P^{\gp}\otimes_{\mathbb{Z}}\mathcal{O}_C[0]$, which is isomorphic to $\overline{M}_{X}^{\gp}\otimes_{\mathbb{Z}}\mathcal{O}_C[0]$ by \Cref{logstratacharts} (iv). The canonical morphism $g^*j^*L_{\D(P)/\mathbb{Z}}\to i^*L_{(X,M_X)/\mathbb{Z}}$ provides a section, proving (i). For (ii), compose the morphism of (i) with the canonical map $L_{B/(X,M_X)}\to s^*i^*L_{(X,M_X)/\mathbb{Z}}[1]$ \cref{Lxdef}.

For (iii), let $h=g\circ s$ and consider the commutative diagram of distinguished triangles
\[
\begin{tikzcd}
(h^*j^*L_{\D(P)/\mathbb{Z}})^{\oplus 2} \arrow[d] \arrow[r] & s^*i^*L_{(X,M_X)/\mathbb{Z}}\oplus h^*L_{\mathcal{S}/\mathbb{Z}} \arrow[r] \arrow[d] & s^*i^*L_{(X,M_X)/\D(P)}\oplus h^*L_{\mathcal{S}/\D(P)} \arrow[d] \\
h^*j^*L_{\D(P)/\mathbb{Z}} \arrow[r] & L_{(B,s^*i^*M_X)/\mathbb{Z}} \arrow[r] & L_{(B,s^*i^*M_X)/\D(P)}
\end{tikzcd}
\]
where the vertical arrows are the sums of the canonical morphisms. Let us compute the cones of the vertical arrows. For the left arrow it is obviously $h^*j^*L_{\D(P)/\mathbb{Z}}[1]$. For the middle arrow, the commutative diagram of distinguished triangles
\[ \begin{tikzcd}
h^*L_{\mathcal{S}/\mathbb{Z}} \arrow[d,equal] \arrow[r] & s^*i^*L_{(X,M_X)/\mathbb{Z}}\oplus h^*L_{\mathcal{S}/\mathbb{Z}} \arrow[r] \arrow[d] & s^*i^*L_{(X,M_X)/\mathbb{Z}} \arrow[d] \\
h^*L_{\mathcal{S}/\mathbb{Z}} \arrow[r] & L_{(B,s^*i^*M_X)/\mathbb{Z}} \arrow[r] & L_{(B,s^*i^*M_X)/\mathcal{S}}
\end{tikzcd} \]
provides a quasi-isomorphism with $L_{B/(X,M_X)}$. For the right arrow, the commutative diagram of distinguished triangles
\[ \begin{tikzcd}
s^*i^*L_{(X,M_X)/\D(P)} \arrow[d,equal] \arrow[r] & s^*i^*L_{(X,M_X)/\D(P)}\oplus h^*L_{\mathcal{S}/\D(P)} \arrow[r] \arrow[d] & h^*L_{\mathcal{S}/\D(P)} \arrow[d] \\
s^*i^*L_{(X,M_X)/\D(P)} \arrow[r] & L_{(B,s^*i^*M_X)/\D(P)} \arrow[r] & L_{(B,s^*i^*M_X)/(X,M_X)}
\end{tikzcd} \]
provides a quasi-isomorphism with $K_B$. Thus, we have a distinguished triangle
\[ h^*j^*L_{\D(P)/\mathbb{Z}}[1]\to L_{B/(X,M_X)}\to K_B \]
which is split by the map of (ii): this can be checked with $(X,M_X)=\D(P)$ and $B=\mathcal{S}$, in which case $K_B=0$ and it is obvious.
\end{proof}

Now \Cref{distinguished} follows from \Cref{canonicalmap} (iii) and the next lemma.

\begin{lemma}\label{tflatiso}
Assume $\overline{M}_{X,x}^{\gp}$ is torsion free and let $P\to M_X$ be a chart that is neat at $x$. Let $K_C$ be the cone of the canonical morphism $g^*L_{\mathcal{S}/\D(P)}\to L_{(C,i^*M_X)/(X,M_X)}$. Then, $H^iK_C=0$ for $i\geq -1$ and the following are equivalent.
\begin{enumerate}[label=\normalfont(\roman*)]
\item $(X,M_X)$ is t-flat at $x$.
\item $(K_C)_x\simeq 0$
\item $H^{-2}(K_C)_x=0$
\item $H^{-2}(x^*K_C)=0$
\end{enumerate}
\end{lemma}
\begin{proof}
By \Cref{strictcotangent}, $K_C$ is quasi-isomorphic to the cone of the morphism of the cotangent complexes of the underlying schemes. Thus, by \cite[\href{https://stacks.math.columbia.edu/tag/09DB}{Tag 09DB}]{stacks-project}, we have $H^iK_C=0$ for $i\geq -1$ and $H^{-2}(K_C)_x\cong\Tor_1^{\mathbb{Z}[P]}(\mathbb{Z},\mathcal{O}_{X,x})$. If (i) holds, then $\Tor_i^{\mathbb{Z}[P]}(\mathbb{Z},\mathcal{O}_{X,x})=0$ for $i>0$ and $(K_C)_x\simeq 0$ by \cite[II, 2.2.3]{cotangent1}, hence (ii)--(iii). (iii) implies $(K_C)_x\cong \tau_{\leq -3}(K_C)_x$ and (iv) follows. Assume (iv). Since $K_C\cong \tau_{\leq -2}K_C$, the spectral sequence for $\Tor$ implies 
\[ H^{-2}(x^*K_C)\cong H^{-2}(K_C)_x\otimes_{\mathcal{O}_{X,x}}k(x)\cong\Tor_1^{\mathbb{Z}[P]}(\mathbb{Z},\mathcal{O}_{X,x})\otimes_{\mathcal{O}_{X,x}}k(x) \]
hence Nakayama's lemma implies $\Tor_1^{\mathbb{Z}[P]}(\mathbb{Z},\mathcal{O}_{X,x})=0$ and (i) follows from \Cref{maxideals}.
\end{proof}

As a consequence of \Cref{distinguished}, we recover a characterisation of log regularity due to Conde-Lago and Majadas \cite{cotangentlogreg}.

\begin{corollary}\label{logregvanishing}
Assume $\overline{M}_{X,x}^{\gp}$ torsion free and let $\bar{x}\in X_{(x)}$ be the closed point. The following are equivalent.
\begin{enumerate}[label=\normalfont(\roman*)]
\item $(X,M_X)$ is t-flat at $x$ and $C(M_{X,x})$ is regular.\footnote{i.e., $(X,M_X)$ is t-regular at $x$ (\Cref{tregdef})}
\item $H^iL_{\bar{x}/(X,M_X)}=0$ for $i\neq -1$.
\item $H^{-2}L_{\bar{x}/(X,M_X)}=0$
\end{enumerate}
In particular, if $M_X$ is saturated, then $(X,M_X)$ is log regular at $x\to X$ if and only if \textup{(iii)} holds.
\end{corollary}
\begin{proof}
Without loss of generality, we may assume $x=\bar{x}$. The distinguished triangle
\[ x^*L_{C/X}\to L_{x/X}\to L_{x/C} \]
provides a triangle
\begin{equation}\label{triangleregular}
x^*K_C\to K_x\to L_{x/C}
\end{equation}
where $K_{?}$ is as in \Cref{canonicalmap} (iii). By \Cref{tflatiso}, $(X,M_X)$ is t-flat if and only if $x^*K_C\simeq 0$ if and only if $H^{-2}(x^*K_C)$ $=$ $0$. On the other hand, recall that $C$ is regular if and only if  $H^iL_{x/C}=0$ for $i\neq -1$ if and only if $H^{-2}L_{x/C}=0$ (\cite[VI, prop. 26]{andre}). So if (i) holds, then $K_x\simeq L_{x/C}$ is concentrated in degree $-1$ and (ii) follows from \Cref{canonicalmap} (iii). Conversely, if (iii) holds, then \Cref{canonicalmap} (iii) implies $H^{-2}K_x=0$, hence $H^{-2}L_{x/C}\to H^{-1}(x^*K_C)$ is injective. But $H^{-1}(x^*K_C)=0$ by \Cref{tflatiso}, so $H^{-2}L_{x/C}=0$, hence $C$ is regular and $H^iL_{x/C}=0$ for $i\neq -1$. In particular, $H^{-3}L_{x/C}=0$ so $H^{-2}(x^*K_C)=0$ and $(X,M_X)$ is t-flat at $x$ by \Cref{tflatiso}, so (i) holds. This proves the equivalence of (i)--(iii) and the final assertion follows from \Cref{logregcharacterisation}.
\end{proof}

\begin{remark}
The fact that \Cref{logregvanishing} is equivalent to the criterion of \cite{cotangentlogreg} is not immediate; in the notation of \Cref{logcotangentremark}, one needs $L_{C/\mathbb{Z}}\cong L_{(C,N_C)/\mathbb{Z}}$, which follows from \cite[4.3]{cotangentlogreg}.
\end{remark}

There is also variant for complete intersections, cf. \cite{cotangentlogreg}. Recall that a locally noetherian scheme is a \emph{complete intersection \emph{(\emph{ci})}} at a point if its local ring at that point is a complete intersection in the sense of \cite[IV\textsubscript{4}, 19.3.1]{ega}. This property is stable under generisation (\cite[\S 3]{avramov}). A locally noetherian scheme is a \emph{local complete intersection \emph{(\emph{lci})}} if it is ci at all of its points (cf. \cite[\href{https://stacks.math.columbia.edu/tag/09Q3}{Tag 09Q3}]{stacks-project}).

\begin{corollary}\label{lcivanishing}
Assume $x\to X$ is the closed point of $X_{(x)}$ and $\overline{M}_{X,x}^{\gp}$ is torsion free. The following are equivalent.
\begin{enumerate}[label=\normalfont(\roman*)]
\item For any factorisation $x\to U\to X$ with $x^*L_{U/X}\simeq 0$ and strict closed immersion $(U,M_X|_U)\to (Z,M_Z)$ with $(Z,M_Z)$ satisfying the equivalent conditions of \Cref{logregvanishing}, the underlying closed immersion $U\to Z$ is regular at the image of $x$.
\item For some factorisation $x\to U\to X$ with $x^*L_{U/X}\simeq 0$ and strict closed immersion $(U,M_X|_U)\to (Z,M_Z)$ with $(Z,M_Z)$ satisfying the equivalent conditions of \Cref{logregvanishing}, the underlying closed immersion $U\to Z$ is regular at the image of $x$.
\item $H^iL_{x/(X,M_X)}=0$ for $i<-2$.
\item $H^{-3}L_{x/(X,M_X)}=0$
\end{enumerate}
If $\dim(X_{(x)},M_X)=\dim X_{(x)}$, then these conditions are also equivalent to
\begin{enumerate}[label=\normalfont(\roman*),resume]
\item $(X,M_X)$ is t-flat at $x$ and $C(M_{X,x})$ is lci.
\end{enumerate}
\end{corollary}

\begin{lemma}\label{completioncotangent}
If $x\to X$ factorises over a morphism $U\to X$ with $x^*L_{U/X}\simeq 0$, then $L_{x/(X,M_X)}\simeq L_{x/(U,M_X|_U)}$. In particular, this holds for $U=\hat{X}_{(x)}$, the completion of the strict localisation $X_{(x)}$ of $X$ at $x$.
\end{lemma}
\begin{proof}
We have a commutative diagram of distinguished triangles
\[ \begin{tikzcd}
x^*L_{(X,M_X)/\mathbb{Z}} \arrow[d] \arrow[r] & L_{x/\mathbb{Z}} \arrow[d,equal] \arrow[r] & L_{x/(X,M_X)} \arrow[d,dotted] \\
x^*L_{(U,M_X|_{U})/\mathbb{Z}} \arrow[r] & L_{x/\mathbb{Z}} \arrow[r] & L_{x/(U,M_X|_{U})}
\end{tikzcd} \]
and the cone of the left vertical map is $x^*L_{(U,M_X|_U)/(X,M_X)}\cong x^*L_{U/X}\simeq 0$, hence the first assertion. For the second, note that $x^*L_{\hat{X}_{(x)}/X_{(x)}}\cong L_{x/x}$ since $\hat{X}_{(x)}\to X_{(x)}$ is flat (\cite[II, 2.2.3]{cotangent1}), and $L_{X_{(x)}/X}\simeq 0$, so $x^*L_{\hat{X}_{(x)}/X}\simeq 0$.
\end{proof}

\begin{proof}[Proof of \Cref{lcivanishing}]
Let $x\to U\to X$ be a factorisation with $x^*L_{U/X}\simeq 0$ and $(U,M_X|_U)\to (Z,M_Z)$ a strict closed immersion with $(Z,M_Z)$ satisfying the equivalent conditions of \Cref{logregvanishing}. Note that such factorisations exist: by \Cref{completioncotangent} we may take $U=\hat{X}_{(x)}$ and then we can take $(Z,M_Z)=\D(P)\times S$, where $S$ is a regular local scheme containing $U$ and $P\to M_X$ is a chart that is neat at the image of $x$. In particular, (i) implies (ii).

We have a distinguished triangle
\[ x^*L_{U/Z}\to L_{x/(Z,M_Z)}\to L_{x/(X,M_X)} \]
and $L_{x/(Z,M_Z)}$ is concentrated in degree $-1$ by \Cref{logregvanishing}. Since $U\to Z$ is a regular closed immersion at the image of $x$ if and only if $H^{-2}(x^*L_{U/Z})=0$ if and only if $H^i(x^*L_{U/Z})=0$ for $i\neq -1$ (\cite[VI, th. 25]{andre}), this implies the equivalence of (i)--(iv).

Now assume $\dim(X_{(x)},M_X)=\dim\mathcal{O}_{X,x}$. Recall that $C=C(M_{X,x})$ is lci if and only if $H^iL_{x/C}=0$ for $i<-2$ if and only if $H^{-3}L_{x/C}=0$ (\cite[VI, prop. 27]{andre}). If (v) holds, then from \Cref{tflatiso} and \cref{triangleregular} we obtain (iv) via \Cref{canonicalmap} (iii). Conversely, if (i) holds, then let $U=\hat{X}_{(x)}$ and $(Z,M_Z)=\D(P)\times S$ where $S$ is a regular scheme containing $U$ and $P\to M_X$ a chart that is neat at $x$; then, $\dim Z=\dim S+\rk P^{\gp}$. Since $\dim X_{(x)}=\dim(X_{(x)},M_X)=\dim C+\rk P^{\gp}$, it follows that $\dim Z-\dim U=\dim S-\dim C$. Thus, since $U\to Z$ is a regular closed immersion, so is $C\times_{X_{(x)}}U\to S$ by \cite[0\textsubscript{IV}, 16.5.6]{ega}. So $C$ is lci and \Cref{depthlemma2} implies $(U,M_X|_{U})$ is t-flat at $x$, hence so is $(X,M_X)$ (\Cref{tlfcomplete}), whence (v).
\end{proof}

\begin{remark}\label{lciremark}
In \cite{cotangentlogreg}, condition (iv) of \Cref{lcivanishing} is taken as the definition of \emph{log complete intersection}. This is a sensible definition as, thanks to the equivalence with (ii), up to localising and completing, there is a regular scheme $S$ and a \emph{log lci}\footnote{in the sense of \cite[4.4.2]{kato-saitoI}} morphism $(X,M_X)\to S$, and $L_{(X,M_X)/S}$ has tor amplitude in $[-1,0]$, just as in the classical case (\cite[III, 3.2.6]{cotangent1}). However, as the next example shows, this condition alone does not imply anything at all about $C(M_{X,x})$.
\end{remark}

\begin{example}\label{nonlciexample}
Let $R$ be a complete regular local ring and $r_1,...,r_n$ elements of its maximal ideal. Consider the scheme $Z=\Spec(R[\mathbb{N}^n])$ with the natural log structure $M_Z$ induced by $\mathbb{N}^n$. Define $X\subset Z$ as the closed subscheme of ideal $(t_1-r_1,...,t_n-r_n)$, where $t_1,...,t_n\in R[\mathbb{N}^n]$ are the images of the basis elements of $\mathbb{N}^n$. The map $R[\mathbb{N}^n]\to R$, sending $t_i$ to $r_i$ for $1\leq i\leq n$, provides an isomorphism $X\cong\Spec(R)$. Thus, $X\subset Z$ is a regular closed immersion and, letting $M_X$ be the inverse image log structure of $M_Z$, condition (ii) of \Cref{lcivanishing} is satisfied at the closed point $x\in X$. In this case, $\mathcal{O}(C(M_{X,x}))=R/(r_1,...,r_n)$ could be any complete noetherian local ring.
\end{example}

\subsection{t-lci log schemes}
\Cref{lcivanishing} (v) and \Cref{lciremark} suggest introducing a version of log complete intersection that includes t-flatness.

\begin{definition}\label{logsyndef}
A fine log scheme $(X,M_X)$ is \emph{t-lci} at a geometric point $x\to X$ if it is t-flat at $x$ and $C(M_{X,x})$ is lci.
\end{definition}

We immediately deduce from \Cref{lcivanishing} and \Cref{ldimineq} the following characterisation.

\begin{proposition}\label{logsynchar}
Assume $x\to X$ is the closed point of $X_{(x)}$ and $\overline{M}_{X,x}^{\gp}$ is torsion free. The following are equivalent.
\begin{enumerate}[label=\normalfont(\roman*)]
\item $(X,M_X)$ is t-lci at $x$.
\item $\dim(X_{(x)},M_X)=\dim X_{(x)}$ and one of the following conditions hold
\begin{enumerate}[label=\normalfont(\alph*)]
\item for any factorisation $x\to U\to X$ with $x^*L_{U/X}\cong 0$ and strict closed immersion $(U,M_X|_U)\to (Z,M_Z)$ with $(Z,M_Z)$ satisfying the equivalent conditions of \Cref{logregvanishing}, the underlying closed immersion $U\to Z$ is regular at the image of $x$
\item $H^{-3}L_{x/(X,M_X)}=0$
\end{enumerate}
\end{enumerate}
\end{proposition}

In the following we use the notion of log syntomic morphism of Nizio{\l} \cite{niziolKlog}.

\begin{proposition}\label{syntomiclemma}
Let $f:(X,M_X)\to (Y,M_Y)$ be a log flat morphism of fine log schemes and $x\to X$ a geometric point.
\begin{enumerate}[label=\normalfont(\roman*)]
\item If $(X,M_X)$ is t-lci at $x$ and $(Y,M_Y)$ is t-flat at $f(x)$, then $(Y,M_Y)$ is t-lci at $f(x)$.
\item If $(X,M_X)$ is t-lci at $x$ and, locally for the fppf topology at $x$, there is a neat Kummer chart $Q\to P$ for $f$, then $(Y,M_Y)$ is t-lci at $f(x)$. In particular, this holds for $f$ Kummer at $x$.
\item If $f$ is a log syntomic morphism of fs log schemes and $(Y,M_Y)$ is t-lci at $f(x)$, then $(X,M_X)$ is t-lci at $x$.
\end{enumerate}
\end{proposition}
\begin{proof}
Recall that the lci property descends under flat morphisms and is also local for the syntomic topology (cf. \cite[1.9.2]{avramov}, \cite[\href{https://stacks.math.columbia.edu/tag/09Q2}{Tag 09Q2}]{stacks-project}).

For (i), it suffices to show that $C(M_{Y,f(x)})$ is lci. Since $C(M_{X,x})$ is lci, this follows from the flatness of $C(M_{X,x})\to C(M_{Y,f(x)})$ (\Cref{flatstrata}). For (ii) it now suffices to note that $(Y,M_Y)$ is t-flat at $f(x)$, as follows from \Cref{klftlf}.

For (iii), first note that $(X,M_X)$ is t-flat at $x$ (\Cref{klftlf}), so it remains to show that $C(M_{X,x})$ is lci and for this it suffices to show that $C(M_{X,x})\to C(M_{Y,f(x)})$ is the strict localisation of a syntomic morphism. If $f$ is strict and $X\to Y$ is syntomic, then this follows from \Cref{logstratacharts} (ii). In general, locally for the syntomic topology $f$ factors over a strict log syntomic morphism $(X,M_X)\to (Y,M_Y)\times_{\D(Q)}\D(P)$ (by definition, \cite[2.1]{niziolKlog}) and so we reduce to the case $(X,M_X)=(Y,M_Y)\times_{\D(Q)}\D(P)$. Moreover, by \cite[2.8]{niziolKlog} (and \Cref{neatcharts}) we can assume $Q\to M_Y$ is exact at $f(x)$ and $Q\to P$ injective. In particular, $Q\to\mathcal{O}_{Y,f(x)}$ is local, hence so is $Q\to P$. In this case, $C(M_{X,x})$ is a strict localisation of $C(M_{Y,f(x)})\times_{\D(Q^*)}\D(P^*)$ (\Cref{logstratacharts} (iii)) and $\D(P^*)$ is syntomic over $\D(Q^*)$ by \cite[2.5]{niziolKlog}, hence $C(M_{X,x})$ is lci.
\end{proof}

\begin{proposition}\label{logsynfield}
Let $(X,M_X)$ be an fs log scheme, with $X$ of finite type over the spectrum $S$ of a field, and $x\to X$ a geometric point. The following are equivalent.
\begin{enumerate}[label=\normalfont(\roman*)]
\item $(X,M_X)$ is t-lci at $x$.
\item $(X,M_X)\to S$ is log syntomic in a neighbourhood of $x$.
\end{enumerate}
\end{proposition}
\begin{proof}
Assume (i). Locally for the \'etale topology at $x$, there is a neat chart $P\to M_X$ (\Cref{neatcharts} (ii)) and then, as in the proof of \Cref{tlflogflat} (ii), the morphism $(X,M_X)\to S\times\D(P)$ is strict log flat. Up to localising further at $x$, there is a closed immersion $X\to Y$ with $Y$ smooth over $S$. Then, the closed immersion $(X,M_X)\to (Z,M_Z):=Y\times\D(P)$ is strict and, by \Cref{logsynchar}, the closed immersion $X\to Z$ is regular at the image of $x$. Thus, $(X,M_X)\to S\times\D(P)$ is strict log syntomic in a neighbourhood of $x$ and (ii) holds. The converse follows from \Cref{syntomiclemma} (iii).
\end{proof}

\bibliographystyle{amsplain}
\bibliography{new1.bib}

\end{document}